\documentclass[12pt]{amsart}
\usepackage{latexsym, graphicx, amscd, amssymb, xypic,mathdots}
\usepackage{float,color}

\numberwithin{equation}{section}

\newcommand{\cA}{\mathcal{A}}

\newcommand{\cM}{\mathcal{M}}

\renewcommand{\cR}{\mathcal{R}}

\newcommand{\bC}{\mathbb{C}}

\newcommand{\bN}{\mathbb{N}}
\newcommand{\bR}{\mathbb{R}}
\newcommand{\bZ}{\mathbb{Z}}

\newcommand{\fg}{\mathfrak{g}}
\newcommand{\fk}{\mathfrak{k}}

\newcommand{\fh}{\mathfrak{h}}
\newcommand{\fm}{\mathfrak{m}}

\newcommand{\fA}{\mathfrak{A}}
\newcommand{\fa}{\mathfrak{a}}
\newcommand{\fd}{\mathfrak{d}}
\newcommand{\fe}{\mathfrak{e}}
\newcommand{\fsl}{\mathfrak{sl}}
\newcommand{\fsu}{\mathfrak{su}}
\newcommand{\fso}{\mathfrak{so}}
\newcommand{\fsp}{\mathfrak{sp}}

\newcommand{\ad}{\mathrm{ad}}
\newcommand{\Ad}{\mathrm{Ad}}

\newcommand{\Aut}{\mathrm{Aut}}
\newcommand{\End}{\mathrm{End}}
\newcommand{\rank}{\mathrm{rank}}

\newcommand{\diag}{\mathrm{diag}}
\newcommand{\Ker}{\mathrm{Ker}}

\newcommand{\Card}{\mathrm{Card}~}
\newcommand{\tr}{\mathrm{tr}}

\newtheorem{thm}{Theorem}[section]

\newtheorem{lm}[thm]{Lemma}
\newtheorem{rem}[thm]{Remark}
\newtheorem{cor}[thm]{Corollary}
\newtheorem{pro}[thm]{Proposition}
\newtheorem{ex}[thm]{Example}

\newtheorem{df}[thm]{Definition}

\renewcommand\i{\sqrt{-1}}
\newcommand{\ii}{ {\scriptstyle\sqrt{-1}}\, }
\newcommand{\hb}{\hat\beta} 
\newcommand{\bb}{\beta} 
\newcommand{\ho}{\hat\omega} 
\newcommand{\st}{\ \vert\ }
\newcommand{\bp}{\begin{pmatrix}}
\newcommand{\ep}{\end{pmatrix}}
\newcommand{\bsp}{\left(\begin{smallmatrix}}
\newcommand{\esp}{\end{smallmatrix}\right)}

\newcommand{\supersec}{ \widehat{\text{Sect}}  }
\newcommand{\tildesupersec}{ \widetilde{\text{Sect}}  }
\newcommand{\sto}{ \text{Sto}  }
\newcommand{\Phiz}{  \Phi^{(0)}  }
\newcommand{\Phii}{  \Phi^{(\infty)}  }
\newcommand{\phiz}{  \phi^{(0)}  }
\newcommand{\phii}{  \phi^{(\infty)}  }
\newcommand{\simple}{\Gamma}
\newcommand{\leftMC}{  \theta_{ \rm{left}  } }
\newcommand{\rightMC}{  \theta_{ \rm{right}  } }
\newcommand{\domainset}{U}
\newcommand{\Psiz}{  \Psi^{(0)}  }
\newcommand{\tPsiz}{  \tilde\Psi^{(0)}  }
\newcommand{\Kz}{  K^{(0)}  }
\newcommand{\Sz}{  S^{(0)}  }

\newcommand*\newbar[1]{%
   \hbox{%
     \vbox{%
       \hrule height 0.5pt 
       \kern-0.2ex
       \hbox{%
         \kern+0.0em
         \ensuremath{#1}%
         \kern-0.1em
       }%
     }%
   }%
}

\begin{document}

\title[Kostant, Steinberg, and Stokes]{Kostant, Steinberg, and the Stokes matrices of the tt*-Toda equations}

\author{Martin A. Guest}
\address{Department of Mathematics, Waseda University, 3-4-1 Okubo, Shinjuku, Tokyo 169-8555, Japan}
\email{martin@waseda.jp}

\author{Nan-Kuo Ho}
\address{Department of Mathematics, National Tsing Hua University, Hsinchu 300, and National Center for Theoretical Sciences, Taipei 106,  Taiwan}
\email{nankuo@math.nthu.edu.tw}

\date{\today}

\maketitle

\begin{abstract}
We propose a Lie-theoretic definition of the tt*-Toda equations for any complex simple Lie algebra $\fg$, based on the concept of 
topological-antitopological fusion which was introduced by Cecotti and Vafa. Our main result concerns the Stokes data of a certain meromorphic connection, whose isomonodromic deformations are controlled by these equations. Exploiting a framework introduced by Boalch, we show that this data has a remarkable structure, which can be described using Kostant's theory of Cartan subalgebras in apposition and Steinberg's theory of conjugacy classes of regular elements.  A by-product of this is a convenient visualization of the orbit structure of the roots under the action of a Coxeter element. As an application, we compute canonical Stokes data of certain solutions of the tt*-Toda equations in terms of their asymptotics.
\end{abstract}

\section{Introduction}\label{intro}

This article presents some new relations between Lie theory and the Stokes Phenomenon for ordinary differential equations. We focus on complex simple Lie groups and some properties of their root systems which were discovered by Kostant and Steinberg. We shall demonstrate how these properties arise naturally from the monodromy data of a certain complex o.d.e.\ 

The classical o.d.e.\ treatment of Stokes matrices was placed into the framework of  Lie algebra-valued meromorphic connections by Boalch, a point of view which has already proved useful in problems originating in physics such as Frobenius manifolds and mirror symmetry (cf.\  \cite{B_inv},\cite{B_imrn},\cite{BT12}).

Our specific motivation --- with similar origins --- is the theory of the tt* equations (topological-antitopological fusion equations), which were introduced by Cecotti and Vafa in \cite{CFIV92},\cite{CV91},\cite{CV93} as a system of nonlinear p.d.e.\ describing massive deformations of topological field theories.  The equations and their solutions also have rich connections to geometry, in particular unfoldings of singularities and quantum cohomology. 

At the time of publication of \cite{CV91},\cite{CV93}, the sinh-Gordon and Tzitzeica equations were the only examples of tt* equations where significant information was available concerning solutions. More than 20 years later, in 
\cite{GL14},\cite{GIL1},\cite{GIL2},\cite{GIL3},\cite{MoXX},\cite{Mo14}, 
the solutions of certain tt* equations of \lq\lq Toda type\rq\rq\ were studied in detail, confirming some of the predictions made on physical grounds by Cecotti and Vafa.  
These tt*-Toda equations (like the sinh-Gordon and Tzitzeica equations) are special cases of the two-dimensional Toda equations, which exist for any complex Lie algebra $\fg$. 

The examples just mentioned are for $\fg=\fsl_n\bC$.
The main objective of this article is to establish some Lie-theoretic aspects of the tt*-Toda equations which will be required in order to generalize the above results to other $\fg$.  Thus, it is the first step in a general treatment of the equations and their solutions in the context of physics and geometry. 
We emphasize that our work is quite different from the existing Lie-theoretic literature on the Toda equations. This is because we do not use the Toda equations directly, but rather an associated system of meromorphic linear o.d.e.\ with irregular singularities. The tt*-Toda equations arise as the equations describing isomonodromic deformations of this meromorphic system.

For general $\fg$, the very definition of the tt*-Toda equations presents some challenges. After giving some background and motivation in section \ref{back}, in section \ref{eqns} we propose a definition which (we believe) represents the essential features of the examples in the literature. Roughly speaking, this imposes three conditions on the Toda equations for $\fg$:

(R) reality condition

(F) Frobenius condition

(S) similarity (or homogeneity) condition

\noindent These are stated precisely in Definition \ref{tt*-Toda}, and in
Definition \ref{hatalpha} we give the associated meromorphic system.
An important ingredient of (R) and (F) is a certain split real form of 
$\fg$, first arising in work of Hitchin, based on Kostant's theory of three-dimensional subalgebras and Cartan subalgebras in apposition.  

To our surprise, the ensuing theory provides geometrical  insight into some classical results on root systems due to Kostant and Steinberg, as well as the famous \lq\lq Coxeter Plane\rq\rq.  It is of course not surprising that any properties of the Toda equations can be described Lie-theoretically, but we believe it is interesting that the tt*-Toda equations bring to life some rather subtle properties of roots which have apparently not been used widely by geometers. 

In sections \ref{data} and \ref{lie} we investigate the space 
 of \lq\lq abstract Stokes data\rq\rq\ for the tt*-Toda equations.  These two sections are based entirely on the symmetries of the equations, without considering the corresponding solutions.  We summarize this briefly next --- a more technical description with precise statements can be found at the beginning of section \ref{lie}.

The meromorphic system has poles of order $2$ at zero and infinity. Locally this equation has $G$-valued fundamental solutions, where $G$ is a Lie group with Lie algebra $\fg$. Classical o.d.e.\ theory, extended by Boalch, then produces \lq\lq Stokes matrices\rq\rq\ which are elements of certain unipotent subgroups of $G$. By the Riemann-Hilbert correspondence, the Stokes data parametrizes 
isomonodromic deformations of the system, hence solutions of the tt*-Toda equations.  It is therefore of great interest to compute this data.

Unfortunately, computations of Stokes matrices for meromorphic o.d.e.\  are generally quite tedious, and dependent on a number of arbitrary choices.  Once the special structure of our situation has been elucidated, however, the computation becomes relatively straightforward, and furthermore we obtain {\em canonical} Stokes data.  It is then easy to recover the classical Stokes matrices and observe how they depend on the relevant choices. 

The main features of this special structure are as follows. We identify a certain element $M^{(0)}\in G$
(Definition \ref{def:qqpi}), from which all Stokes data can be recovered. It can be regarded as an $s$-th root of the monodromy of the meromorphic system, where $s$ is the Coxeter number. To describe it we use a decomposition of the set of roots of $\fg$ due to Kostant and Steinberg, and prove that the components of this decomposition are in one-to-one correspondence with Stokes factors, or, equivalently, singular directions (certain rays in $\bC$ determined by the meromorphic system). There are $2s$ singular directions, denoted $d_1,\dots,d_{2s}$ (see Figure \ref{singulardirection} in section \ref{lie}). The positive roots correspond to $d_1,\dots,d_{s}$ (the positive sector), and the simple roots correspond to $d_1,d_s$ (the head and tail of the positive sector).  

It is the roots corresponding to $d_1,d_2$ which give rise to
$M^{(0)}$ (see Figure \ref{Roots} in section \ref{lie}).  These roots have a purely Lie-theoretic interpretation: they are the positive roots which become negative under the action of a certain Coxeter element, and they constitute a fundamental domain for the action of this Coxeter element on the set of all roots.
Using work of Steinberg, when $G$ is simply-connected, we show that the space of all possible $M^{(0)}$
can be identified with a fundamental domain for the action of $G$ (by conjugation) on the space $G^{\text{reg}}$ of regular elements of $G$.  This gives an intrinsic description of
$M^{(0)}$, and hence of all Stokes data.

All this information can be visualized in the Coxeter Plane (see \cite{Kostant10}), a beautiful object of independent interest which has been described as the \lq\lq most symmetrical\rq\rq\ projection of the root vectors of $\fg$ onto a certain $2$-dimensional plane. We explain this in Appendix \ref{sec:coxeterplane}. 

In section \ref{local} we relate the abstract data $M^{(0)}$ to some actual solutions of the tt*-Toda equations.  We focus on solutions which are defined on a punctured neighbourhood of the origin --- these include the solutions on $\bC^\ast$ which are of primary interest in physics.  For such a solution
we obtain an explicit formula for $M^{(0)}$ in terms of the asymptotics of the solution at the origin (Theorem \ref{explicitQQPi}).  
It turns out that the semisimple part of $M^{(0)}$ lies in the compact real form of $G$. This leads to an identification of the space of all such $M^{(0)}$ (a proper subspace of $\cM_{\Delta_+}$) with a certain convex polytope (Theorem \ref{convexset}). 
We expect (but do not prove here) that this convex polytope also parametrises the solutions on $\bC^\ast$, as in the case 
$G=Sl_{n+1}\mathbb C$. 

The auxiliary meromorphic connection used in section \ref{local} arises in the context of the geometric Langlands correspondence --- see formulas (5.1) and (5.2) of \cite{FG}.

We remark that some of our conclusions for the case $G=Sl_{n+1}\mathbb C$ can be found in \cite{GH}, where we used exclusively the ad hoc notation of \cite{GIL1},\cite{GIL2},\cite{GIL3} and gave \lq\lq proofs by inspection\rq\rq. Here, for general $G$, we go further, and we use more systematic notation and give general proofs.  However, simply putting $G=Sl_{n+1}\mathbb C$ is not sufficient to obtain the results of \cite{GH}.

{\em Acknowledgements:}  The authors thank Eckhard Meinrenken for useful conversations, and for suggesting the relevance of the Coxeter Plane and the article \cite{Kostant10} by Kostant.  They thank Bill Casselman for clarifying the relation between the Coxeter Plane and Coxeter groups (\cite{Ca17} and Appendix \ref{sec:coxeterplane}). They also thank Philip Boalch for discussions, and for pointing out the article \cite{FG} by Frenkel and Gross.
The first author was partially supported by JSPS grant (A) 25247005. He is grateful to the National Center for Theoretical Sciences for excellent working conditions and financial support. The second author was partially supported by MOST grants 105-2115-M-007-006 and 106-2115-M-007-004.

\section{Background on the tt$^*$ and Toda equations}\label{back}

\noindent{\em Background on the tt* equations}

We take as starting point the formulation of Dubrovin \cite{Du93}: the  tt* equations (on a certain two-dimensional domain $U\in\bC$, to be specified later)
are the equations for harmonic maps 
\[
f:\domainset\to GL_n\bR/O_n
\]
together with a \lq\lq similarity condition\rq\rq\ (also specified later).  
It is well known that these equations can be written in zero curvature form, i.e.\
$d\alpha+\alpha\wedge\alpha=0$, for a certain complex $n\times n$-matrix-valued $1$-form $\alpha$ (\cite{BP94}; see also 
\cite{Gu97}).  Indeed, this holds for harmonic maps $\domainset\to G/K$ for any (real) symmetric space $G/K$, and in this context the only requirement on the $\fg^{\mathbb C}$-valued $1$-form $\alpha$ is that it be of the form
\[
\alpha=\tfrac1\lambda \alpha^\prime_{\fm}
+ \alpha_{\fk}
+\lambda \alpha^{\prime\prime}_{\fm}
\]
where $\fg=\fk\oplus\fm$ is a Cartan decomposition of the Lie algebra of the (real) Lie group $G$. Here $\alpha^\prime_{\fm}, \alpha^{\prime\prime}_{\fm}$ take values in $\fm_{1,0}$, 
$\fm_{0,1}$ where 
$\fm^{\mathbb C}
=\fm_{1,0}\oplus\fm_{0,1}$, 
$\alpha_{\fk}$ takes values in $\fk^{\mathbb C}$, and $\lambda$ is a nonzero complex parameter. 

The particular symmetric space  $GL_n\bR/O_n$ can be identified with the set of all  inner products on $\bR^n$. Thus, $f$ may be regarded as a harmonic metric in the trivial vector bundle over $\domainset$ with fibre $\bR^n$ --- it is an example of a harmonic bundle. More precisely, it is an example of a real harmonic bundle. Harmonic bundles were introduced by Hitchin and Simpson in the 1980's, and locally (i.e.\ for trivial bundles) they correspond to harmonic maps into $GL_n\bC/U_n$,  the set of all Hermitian inner products on $\bC^n$.

The theory of real harmonic bundles over compact Riemann surfaces was modified by Hitchin to accommodate bundles with other structure groups, in \cite{Hi92}. Locally these correspond to harmonic maps $f:\domainset\to G/K$ where $G$ is a split real form of a complex semisimple Lie group and $K$ is a maximal compact subgroup of $G$. 

\noindent{\em Background on the Toda equations}

The two-dimensional Toda equations with respect to a complex Lie algebra  were introduced and studied simultaneously by several authors 
(see \cite{MOP81},\cite{Wi81},\cite{DS85},\cite{OT83} for four rather different viewpoints, and more recently \cite{NiRa07}), generalizing much earlier work on the one-dimensional case. Its relation with harmonic maps was soon recognised by differential geometers (cf.\ 
\cite{BPW95},\cite{BP94},\cite{Mc94}). 

To state the equations we need some notation. From now on, we denote by $G$ a {\em complex} simple Lie group of rank $l$, and by $\fg$ its Lie algebra.  Let $\fh$ be a Cartan subalgebra of $\fg$, and let
\[
\fg = \fh \oplus\left(  \oplus_{\alpha\in\Delta} \ \fg_\alpha \right)
\]
be the root space decomposition, 
where $\Delta$ is the set of roots with respect to $\fh$, and where
$\fg_\alpha=\{ \xi\in\fg \st [h,\xi]=\alpha(h)\xi \ \forall h\in\fh\}$. 
By definition $\dim_{\mathbb C}\fh = l$.  

Let $B$ be any positive scalar multiple of the Killing form. Then $B$ is a nondegenerate symmetric bilinear form on $\fg$ which is invariant in the sense that
$B(\Ad(g)h_1,\Ad(g)h_2)=B(h_1,h_2)$ 
and 
$B(\ad(x)h_1,h_2)+B(h_1,\ad(x)h_2)=0$ 
for all $g\in G$, $x\in\fg$, and $h_1,h_2\in\fg$. Here $\Ad,\ad$ denote respectively the adjoint actions of $G,\fg$ on $\fg$. 

For each $\alpha\in\Delta$, we define $H_\alpha\in\fh$ by 
\[
B(h,H_\alpha)=\alpha(h) \quad \forall h\in\fh.
\]
For any choice of simple roots $\Pi=\{\alpha_1,\dots,\alpha_l\}$ we obtain a basis $H_{\alpha_1},\dots,H_{\alpha_l}$ of $\fh$.  
We denote by $\epsilon_1,\dots,\epsilon_l$ the basis of $\fh$ which is dual to $\alpha_1,\dots,\alpha_l$ in the sense that
$\alpha_i(\epsilon_j)=\delta_{ij}$.
For each $\alpha\in\Delta$, the root space $\fg_\alpha$ is one-dimensional. 
It is possible (see \cite{Hel}, chapter 3, Theorem 5.5) to choose basis vectors $e_\alpha\in\fg_\alpha$ so that $B(e_\alpha,e_{-\alpha})=1$ for all $\alpha\in\Delta$. Define the number $N_{\alpha,\beta}$ by 
$[e_\alpha,e_\beta]=N_{\alpha,\beta}e_{\alpha+\beta}$ if $\alpha+\beta$ is a root and $N_{\alpha,\beta}=0$ if not. It follows that $N_{-\alpha,-\beta}=-N_{\alpha,\beta}$. Thus we have 
\begin{equation}\label{brackets}
[e_\alpha,e_\beta]=
\begin{cases}
0 \text{\ if $\alpha+\beta\notin\Delta$}
\\
H_\alpha \text{\ if $\alpha+\beta=0$}
\\
N_{\alpha,\beta} e_{\alpha+\beta} \text{\ if $\alpha+\beta\in\Delta-\{0\}$}
\end{cases}
\end{equation}

The highest root $\psi\in\Delta$ may be written $\psi=\sum_{i=1}^l q_i \alpha_i$ for certain nonnegative integers $q_i$. The Coxeter number of $G$ is then
\[
s=1+\sum_{i=1}^l q_i.
\]
For notational convenience we write $\psi=-\alpha_0$ from now on. 

Now let $k_0,\dots,k_l$ be any nonzero complex numbers.
The two-dimensional (complex, elliptic) Toda equations are the equations
\begin{equation}\label{Toda}
2w_{z \bar z} = -\sum_{i=0}^l k_i e^{ -2\alpha_i(w)} H_{\alpha_i}
\end{equation}
for maps $w:\domainset\to\fh$.  These are often called the \lq\lq affine\rq\rq\  or \lq\lq periodic\rq\rq\ Toda equations.  The  \lq\lq open\rq\rq\ Toda equations are obtained by putting $k_0=0$; these are much easier to solve (see \cite{LS92}) and we do not discuss them. 

Concrete p.d.e.\ are obtained from (\ref{Toda}) by choosing coordinates on $\fh$, and we shall give three popular choices. 

(I) {\em $u_i=\alpha_i(w)=B(w,H_{\alpha_i}), 1\le i\le l$.}

\noindent  Introducing the notation $u_0=\alpha_0(w)=-\sum_{j=1}^l q_j \alpha_j (w) = -\sum_{j=1}^l q_j u_j$, equation (\ref{Toda}) becomes
\[
2(u_i)_{z \bar z} =  -\sum_{j=0}^l k_j  \alpha_i(H_{\alpha_j}) e^{ -2u_j},
\quad
1\le i\le l
\]
(which is valid also for $i=0$). 

(II) {\em $v_i=B(w,\epsilon_i), 1\le i\le l$ where $\alpha_i(\epsilon_j)=\delta_{ij}$.}

\noindent To calculate the right hand side of the Toda equations in terms of the $v_i$ we use
\begin{align*}
B(H_{\alpha_j},\epsilon_i)&=
\begin{cases}
\ \delta_{i,j}\quad 1\le j\le l
\\
-q_i\quad j=0
\end{cases}
\\
\alpha_j(w)&=
\begin{cases}
\ \ \sum_{k=1}^l \alpha_j(H_{\alpha_k}) v_k \quad 1\le j\le l
\\
-\sum_{j,k=1}^l q_j\alpha_j(H_{\alpha_k}) v_k \quad j=0
\end{cases}
\end{align*}
Then equation (\ref{Toda}) becomes
\[
2(v_i)_{z \bar z} =  
q_ik_0 e^{\textstyle  \sum_{j,k=1}^l  2q_j\alpha_j(H_{\alpha_k}) v_k  }
-k_i e^{\textstyle  -\sum_{k=1}^l  2\alpha_i(H_{\alpha_k}) v_k  }.
\]

Using the coroots $\alpha^\ast_j=\tfrac{2}{\alpha_j(H_{\alpha_j})}H_{\alpha_j}$, or the affine Cartan matrix $A$, where $A_{j,i}= \alpha_i(\alpha^\ast_j)$, (I) or (II) may be rewritten in various ways.  For example, by making obvious variable changes, we may write (I) in the form
$(u_i)_{z \bar z} =  \pm
\sum_{j=0}^l A_{j,i} e^{u_j}$, $1\le i\le l$.

(III) {\em $w_i=$ $i$-th coordinate of a faithful representation.}

\noindent For example, each classical matrix Lie algebra has a standard representation in which there is a Cartan subalgebra represented by diagonal matrices.  

\begin{ex}\label{An-0}
{\em
The case $\fg=\fsl_{n+1}\mathbb C=\{ X\in M_{n+1}\mathbb C \st \tr X=0\}$.  

$\fh =\{ \diag(h_0,\dots,h_n) \st h_0,\dots,h_n\in\mathbb C, \sum_{i=0}^n h_i = 0\}$

roots: $x_i-x_j \ (0\le i\ne j\le n)$, where $x_i:\diag(h_0,\dots,h_n)\mapsto h_i$

simple roots: $x_{i-1}-x_i, 1\le i\le n$

$B(X,Y)=\tr XY$

$H_{x_i-x_j}=E_{i,i}-E_{j,j}$, $e_{x_i-x_j}=E_{i,j}$ 

$\psi=x_0-x_n=\alpha_1+\cdots+\alpha_n=-\alpha_0$, $s=n+1$

$\epsilon_i=\left(1-\tfrac{i}{n+1}\right)
(E_{0,0}+\cdots+E_{i-1,i-1})-
\tfrac{i}{n+1}(E_{i,i}+\cdots+E_{n,n})$

\noindent where $E_{i,j}\  (0\le i,j\le n)$ has a $1$ in the $(i,j)$ entry and all other entries zero.  

Writing $w=\diag(w_0,\dots,w_n)$ (the natural coordinates for the standard representation) we find that $u_i=w_{i-1}-w_i$ and $v_i=w_0+\cdots+w_{i-1}$
for $1\le i\le n$.  The three forms of the Toda equations are:

(I)\quad  $2(u_i)_{z \bar z} = k_{i-1} e^{-2u_{i-1}}  - 2k_{i} e^{-2u_{i}} +  k_{i+1} e^{-2u_{i+1}} $,
\quad
$1\le i\le n$

(II)\quad  $2(v_i)_{z \bar z} = k_0 e^{2(v_1+v_n)}  - k_{i} e^{2(v_{i+1}-2v_i+v_{i-1})}$,
\quad
$1\le i\le n$

(III)\quad $2(w_i)_{z \bar z} = k_{i} e^{2(w_i-w_{i-1})}  - k_{i+1} e^{2(w_{i+1}-w_i)}$,
\quad
$0\le i\le n$

\noindent where we take the standard representation of $\fsl_{n+1}\bC$ in (III). Here we regard $w_i$ as periodic in $i$, i.e.\ 
$w_{i+n+1}=w_{i}$. Note that $\sum_{i=0}^n w_i=0$. 
We regard $v_0=0=v_{n+1}$.  
\qed
}
\end{ex}

\noindent{\em Towards the tt*-Toda equations}

In the next section we shall combine the tt* and Toda equations.
We conclude this section by giving some historical remarks and explaining briefly the reason for using a split real form in the tt* equations. The basic motivation (in both physics and differential geometry) comes from variations of Hodge structures (which can be generalized to tt* geometry and the theory of harmonic bundles).   In fact the open tt*-Toda equations do describe examples of variations of Hodge structures (see section 6 of \cite{CV91} and \cite{ChXX}). Physically this is called the \lq\lq conformal\rq\rq\ case. The affine tt*-Toda equations are generalizations of variations of Hodge structures (called nonabelian Hodge structures in \cite{KKP07}). Physically this is called the \lq\lq massive\rq\rq\ case. Concrete examples of tt* equations of \lq\lq Toda type\rq\rq\ were introduced and studied from the physical point of view  in 
\cite{CV91},\cite{CV93}. On the other hand,
Aldrovandi and Falqui \cite{AF95} were the first to use the language of harmonic bundles, in the open (conformal) case and for $G=SL_n\mathbb C$.  Later, and more systematically, Baraglia \cite{Ba15} specialized Hitchin's theory of real harmonic $G$-bundles to the case of harmonic $G$-bundles of Toda-type over compact Riemann surfaces.  For $G=SL_n\mathbb C$ these \lq\lq cyclic harmonic bundles\rq\rq\ were also discovered by Simpson \cite{Si09}. We shall give a precise definition based on all these ideas in the next section.

\section{The tt*-Toda equations}\label{eqns}

As in the discussion of the Toda equations in section \ref{back}, we take $G$ to be a complex simple Lie group of rank $l$, with Lie algebra $\fg$. We choose a Cartan subalgebra $\fh$ and 
system of positive roots  $\Delta_+$.

\subsection{The connection form $\alpha$}\label{sec:connection_alpha}

It is well known that the Toda equations for $w:\domainset\to\fh$ are equivalent to the condition that a certain connection $d+\alpha$ is flat.  Let us review this.

For any nonzero complex numbers $c^{-}_i,c^{+}_i$ with $0\le i\le l$, we define
\[
E_-=\sum_{i=0}^l c^{-}_i e_{-\alpha_i},
\quad
E_+=\sum_{i=0}^l c^{+}_i e_{\alpha_i}
\]
and
\[
\tilde E_-=\Ad(e^w)E_-,
\quad
\tilde E_+=\Ad(e^{-w})E_+.
\]
\begin{df}\label{alphadefinition}
 Let $\alpha=\alpha^\prime dz + \alpha^{\prime\prime}d\bar z$, where 
$\alpha^\prime = w_z + \tfrac1\lambda \tilde E_-$ and
$\alpha^{\prime\prime}= -w_{\bar z} + \lambda \tilde E_+$, and where $\lambda\in\mathbb C^\ast$ is a parameter.
\end{df}

\begin{pro}\label{zcc}
(1) $d\alpha+\alpha\wedge\alpha=0$ for all $\lambda$ if and only if
$2w_{z\bar z} = [ \tilde E_-,\tilde E_+ ]$.
(2) $[ \tilde E_-,\tilde E_+ ]=-\sum_{i=0}^l c^{-}_ic^{+}_i e^{-2\alpha_i(w)} H_{\alpha_i}$.
\end{pro}

\begin{proof} (1) is an elementary calculation. (2) follows immediately from $\Ad(e^w)=e^{\ad(w)}$ and (\ref{brackets}).
\end{proof}

Thus the zero-curvature condition $d\alpha+\alpha\wedge\alpha=0$ is equivalent to
\begin{equation*}
2w_{z \bar z} = -\sum_{i=0}^l c^{-}_ic^{+}_i e^{ -2\alpha_i(w)} H_{\alpha_i}.
\end{equation*}
If we write $k_i=c^{-}_ic^{+}_i$
this gives exactly the Toda equations (\ref{Toda}).

The parameter $\lambda$ is not essential at this point (if $l\ge2$, Proposition \ref{zcc} remains true if we set $\lambda=1$). However it will play an important role very soon.

For the Toda equations, it will be convenient to use the following definition of real form: 

\begin{df}\label{realToda}
Given a real form of the Lie algebra $\fg$, i.e.\ the fixed point set of a conjugate-linear involution on $\fg$, the corresponding real form of the Toda equations is defined by imposing the conditions

(R1) $\alpha_i(w)\in\mathbb R$ for all $i$.

(R2) $\alpha^\prime(z,\bar z,\lambda)\mapsto \alpha^{\prime\prime}(z,\bar z,1/\bar\lambda)$ under the involution.
\end{df}

\noindent
Condition (R1) means that $w$ takes values in the real vector space
\[
\fh_\sharp=\oplus_{i=1}^l\mathbb R H_{\alpha_i}
\]
(note that this is not necessarily a subspace of the given real form of $\fg$).  
Condition (R2) implies that $\alpha\vert_{\vert\lambda\vert=1}$ takes values in the real form of $\fg$. We regard these as conditions as constraints on the function $w$ (and the constants $c^{-}_i,c^{+}_i$).  Depending on the real form, there may be few or no functions $w$ which satisfy these conditions.  We give two examples to illustrate this.

\begin{df}\label{compact} The (standard) compact real form of $\fg$ is the fixed point space $\fg_{\text{\em cpt}}$ of the conjugate-linear involution $\rho$ defined by $\rho(e_\alpha)=-e_{-\alpha}$,
$\rho(H_\alpha)=-H_{\alpha}$
for all $\alpha\in\Delta$.  
\end{df}

An $\mathbb R$-basis of $\fg_{\text{cpt}}$ is given by $\ii H_{\alpha_i}$ ($1\le i\le l$), 
$\ii(e_\alpha+e_{-\alpha})$, $e_\alpha-e_{-\alpha}$ ($\alpha\in\Delta^+$).
Condition (R2) holds if $c^{+}_i=-\newbar{c^{-}_i}$ for all $i$.  Then all $k_i=c^{-}_ic^{+}_i$ are negative.  This is the real form of the Toda equations studied most widely in differential geometry.

\begin{df}\label{split}
The (standard) split real form of $\fg$ is the fixed point 
space $\fg_{\text{\em split}}$ of the conjugate-linear involution $\theta$ defined by $\theta(e_\alpha)=e_{\alpha}$,
$\theta(H_\alpha)=H_{\alpha}$
for all $\alpha\in\Delta$.  
\end{df}

An $\mathbb R$-basis of $\fg_{\text{split}}$ is given by $H_{\alpha_i}$ ($1\le i\le l$), 
$e_\alpha$ ($\alpha\in\Delta$).  In this case condition (R2) is not satisfied with respect to $\theta$ for any $w$, hence there is no corresponding real form of the Toda equations.   However, there are other split real forms, and we shall explain how to find a suitable one, after looking at a concrete example.

\begin{ex}\label{An-1} (Example \ref{An-0} continued)
{\em 
The case $\fg=\fsl_{n+1}\mathbb C$.  
Condition (R1) simply means that $w=\diag(w_0,\dots,w_n)$ with all $w_i$ real.  The compact real form $\fsu_{n+1}$ is given by $\rho(X)=-X^\ast$, and the real Toda equations are
\[
2(w_i)_{z \bar z} = k_{i} e^{2(w_i-w_{i-1})}  - k_{i+1} e^{2(w_{i+1}-w_i)}
\]
with all $k_i$ negative. By adding constants to the functions $w_i$ we can make $k_i=-1$ for all $i$. Thus, the particular values of the real numbers $k_i$ are not important, just their sign.

The standard split real form $\fsl_{n+1}\mathbb R$ is given by $\theta(X)=\bar X$. As we have seen, there are no corresponding Toda equations.  On the other hand, the split real form  given by $\chi(X)=\Delta \bar X \Delta$, where
\begin{equation}\label{delta}
\Delta=
\begin{pmatrix}
  & & 1 \\
  & \iddots \, & \\
1 & &
\end{pmatrix},
\end{equation}
will turn out to be a suitable candidate for the harmonic bundle equations or the tt*-Toda equations.  Conditions (R1),(R2) are be satisfied if the $w_i$ are real and satisfy the additional conditions $w_i+w_{n-i}=0$, and if we take (for example) all $k_i=1$.
In this situation the real Toda equations are
\begin{equation}\label{Ce-Va-Toda}
2(w_i)_{z \bar z} =  e^{2(w_i-w_{i-1})}  -  e^{2(w_{i+1}-w_i)},
\quad
w_i+w_{n-i}=0.
\end{equation}
These equations, and the formula $\chi(X)=\Delta \bar X \Delta$, were derived on physical grounds in \cite{CV91}, section 6, and from the viewpoint of harmonic maps in \cite{GL14}, section 2.  It should be noted that the fixed point set of $\chi$ is isomorphic to the standard split real form
$\fsl_{n+1}\mathbb R$ --- indeed it is a general fact that all split real forms are conjugate.
Thus, although the split real forms are all equivalent, the form of the Toda equations is very sensitive to the particular choice (relative to the initial choice of compact real form and Cartan subalgebra).
\qed
}
\end{ex}

For general $G$, a suitable split real form has been constructed in another context by Hitchin \cite{Hi92} and Baraglia \cite{Ba15}, based on work of Kostant \cite{Ko59}.   A recent and more detailed treatment can be found in \cite{La17}. We review this next.

First some properties of principal three-dimensional subalgebras (TDS) from \cite{Ko59} are needed.  We put
\[
x_0=\sum_{i=1}^l \epsilon_i = \sum_{i=1}^l r_i H_{\alpha_i}
\]
and
\[
e_0=\sum_{i=1}^l a_i e_{\alpha_i},
\quad
f_0=\sum_{i=1}^l (r_i/a_i) e_{-\alpha_i}.
\]
The real numbers $r_1,\dots,r_l$ are determined by the choice of simple roots, and $a_1,\dots,a_l$ are (arbitrary) nonzero complex numbers.  
Direct computation gives
\begin{equation}\label{tds}
[x_0,e_0]=e_0,\quad
[x_0,f_0]=-f_0,\quad
[e_0,f_0]=x_0.
\end{equation}
In the terminology of \cite{Ko59}, $x_0,e_0,f_0$ generate a principle TDS. This is isomorphic to 
$\fsl_2\mathbb C$ and its adjoint action on $\fg$ decomposes into irreducible representations $V_1,\dots,V_l$ of dimensions $2m_1+1,\dots,2m_l+1$. The positive integers $m_i$ are called the exponents of $G$, and they satisfy
$1=m_1<m_2\le \cdots \le m_{l-1}<m_l=s-1$ (unless $\fg=\fsl_2\mathbb C$, in which case $l=1$ and  $m_1=1$).
Let $u_1,\dots,u_l$ be highest weight vectors for $V_1,\dots,V_l$.  We may take $u_1=e_0$ and $u_l=e_\psi$ (but in general $u_i$ cannot easily be expressed in terms of the root vectors).  It is known that $u_1,\dots,u_l$ are eigenvectors of $\ad x_0$ with eigenvalues $m_1,\dots,m_l$.

Next let 
\[
z_0=e_0+e_{-\psi}.
\]
Kostant shows that $z_0$ is semisimple, and regular in the sense that its centralizer
\[
\fg^{z_0}=\{ x\in\fg \st [x,z_0]=0 \} 
\]
has dimension $l$, and in fact is a Cartan subalgebra of $\fg$.   It is in {\em apposition} (page 1018 of \cite{Ko59}) to the original Cartan subalgebra $\fh$ with respect to the {\em principal element}
\[
P_0=e^{2\pi\i x_0/s} \in G.
\]
The linear transformation 
\begin{equation}\label{tau-def}
\tau=\Ad P_0
\end{equation}
preserves $\fg^{z_0}$.  
It is clear from the definition of $\tau$ that
\begin{equation*}
\tau (e_\alpha)= e^{2\pi\i\alpha(x_0)/s} e_\alpha
= e^{2\pi\i\,\text{ht}\,\alpha\,/s} e_\alpha,
\end{equation*}
where we define the height of a root by
$\text{ht}(\sum_{i=1}^l n_i \alpha_i)=\sum_{i=1}^l n_i$.
The vectors $u_1,\dots,u_l$ are eigenvectors of $\tau$ with eigenvalues $e^{2\pi\i m_1/s},\dots,e^{2\pi\i m_l/s}$.

Let us now take $a_i=\sqrt{r_i}$.
Hitchin \cite{Hi92} (see also section 2 of \cite{La17}) shows that there is a $\mathbb C$-linear involution $\sigma:\fg\to\fg$ which is uniquely defined by the conditions 
\begin{equation}\label{sigma-def}
\sigma(u_i)=-u_i \ (i=1,\dots,l), \ \sigma(f_0)=-f_0.
\end{equation}
Uniqueness is clear as a homomorphism $\sigma$ necessarily satisfies 
$\sigma((\ad\,f_0)^k u_i)= (-1)^{k+1} (\ad\,f_0)^k u_i$ and the elements $(\ad\,f_0)^k u_i$ span $\fg$.  Thus the main point is that the map $\sigma$ defined by $\sigma((\ad\,f_0)^k u_i)= (-1)^{k+1} (\ad\,f_0)^k u_i$ is actually a homomorphism.  This is Proposition 6.1 (1) of \cite{Hi92}.

Using $\sigma$, we introduce the conjugate-linear involution 
\begin{equation}\label{chi-def}
\chi=\sigma\rho
\end{equation}
where $\rho$ defines the standard compact real form (Definition \ref{compact}).
We have $\sigma\rho = \rho\sigma$ since this holds on each $V_i$. 
It can be shown that $\chi$ defines a split real form --- this is Proposition 6.1 (2) of \cite{Hi92}.  

We shall need to know more about $\sigma$.  A first observation is that
\[
\sigma(\fh)=\fh
\]
because the $(\ad\,f_0)^{m_i} u_i$ ($1\le i\le l$) are a basis of $\fh$, and 
$\sigma((\ad\,f_0)^{m_i} u_i)=(-1)^{m_i+1} (\ad\,f_0)^{m_i} u_i$.  

Next, since $[x,\sigma(e_\alpha)]=\sigma[\sigma(x),e_\alpha]=
\alpha(\sigma(x)) \sigma(e_\alpha)$ for all $x\in\fh$, we see that $\sigma(e_\alpha)$ is a root vector for $\alpha\circ\sigma$, and hence a multiple of 
$e_{\alpha\circ\sigma}$.  Let us write $\sigma(e_\alpha)=t_\alpha e_{\alpha\circ\sigma}$.   
For any simple root $\alpha_i$, the root  $\alpha_i\circ\sigma$ must also be simple. Namely, if $\alpha_i\circ\sigma = \sum_{j=1}^l c^i_j \alpha_j$, then
$\alpha_i=\alpha_i\circ\sigma\circ\sigma=
\sum_{j,k=1}^l c^i_j  c^j_k\alpha_k$, where the $c^i_j$ are either
all nonnegative integers or all nonpositive integers.  It follows that
\begin{equation}\label{autonu}
\alpha_i\circ\sigma = \alpha_{\nu(i)}
\end{equation}
for some permutation $\nu$ of $\{1,2,\dots,l\}$.  

\begin{pro}\label{sigma}  We have 
$\sigma( e_{\alpha_i} )= - e_{\alpha_{\nu(i)}}$,
$\sigma( e_{-\alpha_i} )= - e_{-\alpha_{\nu(i)}}$ 
for $i=1,2,\dots,l$ and also
$\sigma( e_{\psi} )=-e_{\psi}, \sigma( e_{-\psi} )=-e_{-\psi}$.
\end{pro}

\begin{proof}  From $\sigma(u_1)=-u_1$, where $u_1=\sum_{i=1}^l \sqrt{r_i} e_{\alpha_i} $, we have
\[
-\sum_{i=1}^l \sqrt{r_i} e_{\alpha_i} = \sum_{i=1}^l \sqrt{r_i} \sigma(e_{\alpha_i})
= \sum_{i=1}^l \sqrt{r_i} t_{\alpha_i} e_{\alpha_i\circ \sigma}= \sum_{i=1}^l \sqrt{r_i} t_{\alpha_i} e_{\alpha_{\nu(i)}}.
\]
It is known that ${r_i}=
{r_{\nu(i)}}$ (cf.\ Remark \ref{outer}), so we deduce that $t_{\alpha_i}=-1$ for all $i$. This gives the first statement $\sigma( e_{\alpha_i} )= - e_{\alpha_{\nu(i)}}$. The statement $\sigma( e_{\psi} )=-e_{\psi}$ holds by definition of $\sigma$, since $e_{\psi}=u_l$.  The two remaining statements follow immediately from these by applying $\rho$ and using the fact that $\sigma$ commutes with $\rho$.
\end{proof}

\begin{rem}\label{outer} If 
$\fg\ne \fa_l,\fd_{2m+1},\fe_6$ then the above formula
$\sigma((\ad\,f_0)^{m_i} u_i)=(-1)^{m_i+1} (\ad\,f_0)^{m_i} u_i$
shows that $\sigma$ is the identity map on $\fh$, because for such $\fg$ all exponents $m_i$ are known to be odd.
In this case, $\nu$ is the identity permutation.  If 
$\fg= \fa_l,\fd_{2m+1},\fe_6$, it is known (see \cite{Hi92}, Remark 6.1) 
that $\nu$ is induced by an outer automorphism of $\fg$ corresponding to a symmetry of the Dynkin diagram. 
\end{rem}

The connection form $\alpha$ has the following symmetries.

\begin{pro}\label{alphasymmetries}
Let $w:\domainset\to \fh_\sharp$ satisfy the condition 
\begin{equation}\label{antisymmetry}
\sigma(w)=w.
\end{equation}
Choose complex numbers $c^{-}_0,\dots,c^{-}_l$, $c^{+}_0,\dots,c^{+}_l$ such that
\begin{equation}\label{balancing}
c^{-}_i=c^{-}_{\nu(i)}, c^{+}_i=c^{+}_{\nu(i)}, \newbar{c^{-}_i}=c^{+}_{\nu(i)}
\end{equation}
for all $i$, where $\nu(i)$ is defined in (\ref{autonu}) for $i=1,2,\dots,l$ and we put $\nu(0)=0$. Then the connection form $\alpha$ of Definition \ref{alphadefinition} satisfies

(i) $\tau(\alpha(\lambda))=\alpha(e^{2\pi\i/s}\lambda)$

(ii) $\sigma(\alpha(\lambda))=\alpha(-\lambda)$

(iii) $\chi(\alpha^\prime(\lambda))= \alpha^{\prime\prime}(1/\bar\lambda)$

\noindent for all $\lambda\in\mathbb C^\ast$.
\end{pro}

\begin{proof} (i) This follows from the fact that $\Ad e^w$ commutes with $\tau=\Ad P_0$, and $\tau(e_{\alpha_i})=e^{2\pi\i/s}\, e_{\alpha_i}$, 
$\tau(e_{-\alpha_i})=e^{-2\pi\i/s} e_{-\alpha_i}$.
(ii) Since $\Ad e^w=e^{\ad w}$, and by (\ref{antisymmetry}) $\sigma\circ\ad w=
\ad w\circ\sigma$, we have $\sigma\circ\Ad e^w = \Ad e^w \circ\sigma$.
By Proposition \ref{sigma} and (\ref{balancing}), $\sigma(E_-)=-E_-$ and
$\sigma(E_+)=-E_+$.  Combining these, we obtain (ii).
(iii) Since $\Ad e^w=e^{\ad w}$ and $\rho(w)=-w$, 
where $\rho$ is as in Definition \ref{compact},
we have $\rho\circ\Ad e^w = \Ad e^{-w} \circ\rho$.  Hence
$\chi\circ\Ad e^w = \Ad e^{-w} \circ\chi$.  
Using Proposition \ref{sigma} and (\ref{balancing}), $\chi(E_-)=E_+$ and
$\chi(E_+)=E_-$.  Combining these, we obtain (iii).
\end{proof}

Condition (i) is the fundamental \lq\lq cyclic\rq\rq\ property of the Toda equations. 
The condition $\sigma(w)=w$ is needed here mainly to establish (iii), i.e.\  compatibility with the real form involution $\chi$.  However, compatibility with $\sigma$ is a property of independent interest in tt* geometry, which might be called the \lq\lq Frobenius condition\rq\rq. For this reason we state (ii) as a separate observation.

\subsection{The connection form $\hat\alpha$}

The final piece in the construction of the tt*-Toda equations is the similarity condition, together with an associated connection form $\hat\alpha$.  

To motivate this, we begin with an informal discussion, following 
\cite{Du93},\cite{CV91}.
The similarity condition may be stated as
\begin{equation}\label{homogeneity}
\left(
\lambda\tfrac{\partial}{\partial \lambda}
+z\tfrac{\partial}{\partial z}
-\bar z\tfrac{\partial}{\partial \bar z}
\right)
F=0
\end{equation}
for a suitable $G$-valued function $F(z,\bar{z},\lambda)$ such that $\alpha=F^{-1}\frac{\partial F}{\partial z}dz+F^{-1}\frac{\partial F}{\partial \bar{z}}d\bar{z}$ (cf.\ 
\cite{Du93}, formula (2.19)).  Here we are using classical matrix notation for $F$. From the explicit form of $\alpha$, it follows that
$(z\tfrac{\partial}{\partial z}
-\bar z\tfrac{\partial}{\partial \bar z})w=0$, which means that $w$ depends only on the radial variable
\[
x=\vert z\vert.
\]  
Since $2w_{z\bar z} = \tfrac12(w_{xx}+\tfrac1x w_x)$, 
the natural (maximal) domain of a solution is 
$\mathbb C^\ast=\mathbb C-\{0\}$.  For physical reasons (completeness of the renormalization group flow) smoothness on this domain will be a requirement. From this discussion, and bearing in mind also Proposition \ref{alphasymmetries}, we arrive at the following definition:

\begin{df}\label{tt*-Toda} {\bf (The tt*-Toda equations)}
Let $\fg$ be a complex simple Lie algebra.
Let $\fg_{\mathbb R}$ be the split real form of $\fg$ defined by $\chi=\rho\sigma$. The tt*-Toda equations (for $\fg$) are the Toda equations which are 

{\em (R)} real with respect to $\fg_{\mathbb R}$  (see Definition \ref{realToda}) 

\noindent where $w:\mathbb C^\ast\to \fh_\sharp$ satisfies the additional conditions

{\em (F)} $\sigma(w)=w$

{\em (S)} $w=w(\vert z\vert)$.
\end{df}

More prosaically, 
the tt*-Toda equations are the Toda equations 
\begin{equation}\label{ttt}
2w_{z \bar z} = -\sum_{i=0}^l k_i e^{ -2\alpha_i(w)} H_{\alpha_i}
\end{equation}
for $w:\mathbb C^\ast\to \fh_\sharp$ such that
$\sigma(w)=w$ and $w=w(\vert z\vert)$,
with $k_i=k_{\nu(i)}$, $k_i>0$ for all $i$.  But this would be an unhelpful definition as it obscures the Lie-theoretic origins.

The Frobenius condition $\sigma(w)=w$ is nontrivial only for $G$ of type $A_n$, $D_{2n+1}$, $E_6$.  Whether this is a desirable feature (for physical reasons), or a defect (which might be remedied by choosing a different real form), we leave for the reader's consideration.

\begin{ex}\label{An-2} (Example \ref{An-1} continued)
{\em
The case $\fg=\fsl_{n+1}\mathbb C$. Let us see how the split real form of Example \ref{An-1} and the equations (\ref{Ce-Va-Toda}) arise from Definition \ref{tt*-Toda}.

$x_0=\diag(\tfrac n2,\tfrac n2-1,\dots,-\tfrac n2)
=\sum_{i=1}^n \tfrac12 i(n-i+1) H_{\alpha_i}$

$P_0=\diag(e^{n\pi\i/s},e^{(n-2)\pi\i/s},\dots,e^{-n\pi\i/s})$

$u_1=\sum_{i=0}^{n-1} a_{i+1} E_{i,i+1}$ with $a_i=a_{n-i+1}$ and $a_i\ne 0$ for all $i$; then $u_2,\dots,u_n$ are determined up to scalar multiples by the conditions $[u_i,u_j]=0$ (the nonzero entries of $u_i$ are on the $i$-th superdiagonal)

$\sigma(X)=-\Delta X^t \Delta$ and $\nu(i)=n-i+1$ ($1\le i\le n$), where $\Delta$ denotes the matrix of (\ref{delta})

$\chi(X)=\Delta \bar X \Delta$

\noindent The condition $\sigma(w)=w$ is $w_i+w_{n-i}=0$ for $0\le i\le n$. 
Let us take $c^{-}_i=c^{-}_{n-i+1}$, $c^{+}_i=c^{+}_{n-i+1}$, 
 $\newbar{c^{-}_i}=c^{+}_{n-i+1}$ for $1\le i\le n$ and
 also $\newbar{c^{-}_0}=c^{+}_{0}$  in the definition of $\alpha$. Then we obtain the tt*-Toda equations
\[
2w_{z \bar z} = -\sum_{i=0}^l k_i e^{ -2\alpha_i(w)} H_{\alpha_i},
\]
where $k_i=\vert c^{-}_i\vert^2>0$ and $k_i=k_{n-i+1}$.
When all $k_i=1$ we obtain (\ref{Ce-Va-Toda}).
\qed
}
\end{ex}

To motivate the definition of $\hat\alpha$, we continue the informal discussion from above.
The similarity condition (\ref{homogeneity}) leads to a new connection form
\begin{align*}
\hat\alpha&= F^{-1}\tfrac{\partial F}{\partial \lambda} d\lambda
\\
&=F^{-1}\left(   
-z\tfrac{\partial F}{\partial z}
+\bar z\tfrac{\partial F}{\partial \bar z}
\right) \tfrac{d\lambda}{\lambda}
\\
&=(-z \alpha^\prime + \bar z \alpha^{\prime\prime}) \tfrac{d\lambda}{\lambda}
\\
&=
\left(
-\tfrac{z}{\lambda}\tilde E_- - xw_x + \lambda\bar z \tilde E_+
\right)\tfrac{d\lambda}{\lambda}
\end{align*}
where we have used $zw_z+\bar zw_{\bar z} = xw_x$ at the last step.
By construction we have 
$\alpha+\hat\alpha=F^{-1}\frac{\partial F}{\partial z}dz+F^{-1}\frac{\partial F}{\partial \bar{z}}d\bar{z}+F^{-1}\frac{\partial F}{\partial \lambda}d\lambda$.

As pointed out by Cecotti-Vafa and Dubrovin,
the meromorphic $\fg$-valued $1$-form $\hat\alpha$ on $\mathbb C\sqcup\infty$ is fundamental in the study of the tt*-Toda equations. 
It has poles of order $2$ at zero and infinity. 
Assuming the existence of $F$, the combined connection $d+\alpha+\hat\alpha$ is flat, and then a standard argument (see chapter 1 of \cite{FIKN06}) shows that the monodromy data of $\hat\alpha$, in particular the Stokes matrices, at these poles does not depend on $x$. 

Putting $\mu=\lambda x/z$,
we obtain a more symmetrical form, which we take as the formal definition of $\hat\alpha$:

\begin{df}\label{hatalpha}{\bf (The meromorphic connection)} 
$\hat\alpha=
\left(
-\tfrac{x}{\mu}\tilde E_- - xw_x + \mu x \tilde E_+
\right)\tfrac{d\mu}{\mu}$,
where $x=\vert z\vert$ and $\mu=\lambda x/z$.
\end{df}

 In section \ref{local} we shall establish the existence of $F$ exists for a particular class of solutions (Proposition \ref{gLRG}).  For the time being we note that Definition \ref{hatalpha} makes sense without reference to $F$, and that the isomonodromy property can also be established without using $F$, as follows.

\begin{lm}\label{zcc2} 
Let $\hat\alpha$ be as in Definition \ref{hatalpha}, and let $\alpha_r$ be the  $\fg$-valued $1$-form on $\mathbb R_{>0}$ defined by
$
\alpha_r = 
\left(
\tfrac{1}{\mu}\tilde E_-  + \mu \tilde E_+ \right)dx.
$
Then: the connection $d+\hat\alpha + \alpha_r$ is flat if and only if $w$ satisfies
the equation $\tfrac12(w_{xx}+\tfrac1x w_x)=[ \tilde E_-,\tilde E_+ ]$.
\end{lm}

\begin{proof} Elementary calculation.
\end{proof}

Thus a solution of the tt*-Toda equations always gives a family (depending on $x\in \mathbb R_{>0}$) of {\em isomonodromic deformations} of $\hat\alpha$.  Evidently this property holds even if the solution is defined only on some open subset of $\mathbb R_{>0}$.

\begin{pro}\label{hatalphasymmetries}
Under the same conditions as Proposition \ref{alphasymmetries}, the connection form $\hat\alpha$ of Definition \ref{hatalpha} satisfies

(i) $\tau(\hat\alpha(\mu))=\hat\alpha(e^{2\pi\i/s}\mu)$

(ii) $\sigma(\hat\alpha(\mu))=\hat\alpha(-\mu)$

(iii) $\chi(\hat\alpha(\mu))= \hat\alpha(1/\bar\mu)$ 

(iv) $\theta(\hat\alpha(\mu))=\hat\alpha(\bar\mu)$ 
(when the $c^{\pm}_i$ are real)

\noindent for all $\mu\in\mathbb C^\ast$.
\end{pro}

\begin{proof}  As (i),(ii),(iii) are similar to the corresponding proofs in Proposition \ref{alphasymmetries}, we just give the proof of (iv).
Since $\theta\vert_{\fh_\sharp}$ is the identity, we have $\theta\circ\ad w=
\ad w\circ\theta$, hence $\theta\circ\Ad e^w = \Ad e^w \circ\theta$.
By definition of $\theta$ (and the conditions on $c^{-}_i,c^{+}_i$) we have $\theta(E_-)=E_+$ and
$\theta(E_+)=E_-$.  Combining these, we obtain (iv).
\end{proof}

These symmetries of the connection form $\hat\alpha$ will play an important role from now on, as they restrict severely its monodromy data.

\section{Definition of the Stokes data}\label{data}

In this section we shall investigate the Stokes data, using the Lie-theoretic approach of Boalch \cite{B_imrn}.  To facilitate comparison with that paper, we shall use the connection $d-\hb$, where
\begin{equation}\label{betahatmu}
\hb=
\left(
-\tfrac{x}{\mu}\tilde E_+ - xw_x + \mu x \tilde E_-
\right)\tfrac{d\mu}{\mu},
\end{equation}
instead of $d+\hat\alpha$.  This gives the same tt*-Toda equations because of the following version of Lemma \ref{zcc2}, whose proof is equally elementary:

\begin{lm}\label{zcc3}
Let $\hb$ be as in (\ref{betahatmu}), and let $\beta_r$ be the  $\fg$-valued $1$-form on $\mathbb R_{>0}$ defined by
\[
\beta_r = 
\left(
\tfrac{1}{\mu}\tilde E_+  + \mu \tilde E_- \right)dx
\]
Then: the connection $d-\hb - \beta_r$ is flat if and only if $w$ satisfies
the equation $\tfrac12(w_{xx}+\tfrac1x w_x)=[ \tilde E_-,\tilde E_+ ]$.
\end{lm}

Using $\hb$ instead of $\hat\alpha$ corresponds to using the connection form 
\[
\bb=(w_z + \tfrac1\lambda \tilde E_+)dz+
(-w_{\bar z} + \lambda \tilde E_-)d\bar z
\]
instead of $\alpha$.  We have
$d\bb-\bb\wedge \bb=0$ (i.e.\ $d-\bb$ is flat) if and only if
$2w_{z\bar z} = [ \tilde E_-,\tilde E_+ ]$.  Thus, as far as the tt*-Toda equations are concerned, it does not matter which choice is made.  We used $\alpha$ in the previous section for compatibility with the cited literature on the Toda equations. 

\begin{rem}
Although the connection $d+\hat\alpha$ is used in 
\cite{GIL1},\cite{GIL2},\cite{GIL3}, for $G=Sl_{n+1}\mathbb C$, the Stokes analysis was actually performed there for the dual connection $d-\hat\alpha^t$.  When $G=Sl_{n+1}\mathbb C$ and $c^{-}_i=c^{+}_i=1$ we have $\hb=\hat\alpha^t$. Thus, the use of $\hb$ here is consistent with those references as well.
\end{rem}

As in \cite{GIL1},\cite{GIL2},\cite{GIL3}, it will be convenient to replace the variable $\mu$ by
\[
\zeta=\mu/x 
\]
for the Stokes analysis.
Then
\begin{equation}\label{betahat}
\hb=
\left(
-\tfrac{1}{\zeta^2}\tilde E_+ - \tfrac{1}{\zeta}xw_x + x^2 \tilde E_-
\right)d\zeta.
\end{equation}

The differential equation for covariant constant sections of $d-\hb$ (in the trivial vector bundle over the Riemann sphere $\mathbb C\sqcup\infty$) is meromorphic with poles of order $2$ at $\zeta=0,\infty$. 
When $G$ is a matrix group this equation may be written as
\begin{equation}\label{merom}
\tfrac{d\Psi}{d\zeta}=
\left(
-\tfrac{1}{\zeta^2}\tilde E_+ - \tfrac{1}{\zeta}xw_x + x^2 \tilde E_-
\right)
\Psi
\end{equation}
where $\Psi$ is the \lq\lq fundamental solution matrix\rq\rq.  For general $G$ we have 
$\Psi^\ast \rightMC = \hb$ where $\rightMC$ is the right Maurer-Cartan form. 

In analogy with Proposition \ref{hatalphasymmetries} we can establish the following symmetries of $d-\hb$:

\begin{pro}\label{hatbetasymmetries}
Under the same conditions as Proposition \ref{alphasymmetries}, the connection form $\hb$ of (\ref{betahat}) satisfies

(i) $\tau(\hb(\zeta))=\hb(e^{-2\pi\i/s}\zeta)$

(ii) $\sigma(\hb(\zeta))=\hb(-\zeta)$

(iii) $\chi(\hb(\zeta))= \hb(1/(x^2\bar\zeta))$ 

(iv) $\theta(\hb(\zeta))=\hb(\bar\zeta)$ (when the $c^{\pm}_i$ are real)

\noindent for all $\zeta\in\mathbb C^\ast$.  We recall that $\tau,\sigma,\chi$ and $\theta$ were defined in equations
(\ref{tau-def}),(\ref{sigma-def}),(\ref{chi-def}) and Definition \ref{split}.
\end{pro}

Classical o.d.e.\ theory for systems with an irregular pole --- such as (\ref{merom}) --- is based on two basic results: (1) one can choose a formal series solution in a neighbourhood of such a pole; (2) on each \lq\lq good\rq\rq\ sector (Stokes sector) centred at that pole, there exists a unique holomorphic solution whose asymptotic expansion is the chosen formal solution.  Details of the classical point of view can be found in chapter 1 of \cite{FIKN06}.  

The theory was adapted to the situation where the meromorphic form takes values in the Lie algebra $\fg$ of a complex reductive Lie group $G$, in \cite{B_imrn}. We shall apply this version of the theory, with minor modifications, to $\hb$. In our treatment of the tt*-Toda equations, $G$ will be a complex simple Lie group.

The reality conditions (iii),(iv) in Proposition \ref{hatbetasymmetries} relate the Stokes data at $\zeta=0$ to the Stokes data at $\zeta=\infty$ (cf.\ \cite{GIL2}, section 2).  For this reason it suffices to consider the Stokes data at $\zeta=0$.

\begin{lm}\label{lm:formalsolution} At $\zeta=0$,
equation (\ref{merom}) has a unique \lq\lq\,$G$-valued\rq\rq\   formal fundamental solution of the form
$ 
\Psiz_f(\zeta)=e^{-w}(I+\sum_{k\geq 1}\psi_k \zeta^k)e^{\frac{1}{\zeta}E_+}. 
$
\end{lm}

Note that $e^{-w}$ and $e^{\frac{1}{\zeta}E_+}$ are indeed $G$-valued. To say that the formal series
$I+\sum_{k\geq 1}\psi_k \zeta^k$ is $G$-valued makes sense whenever $G$ is a complex algebraic group, as in our situation. We use $I$ (rather than $e$) for the identity element of $G$ to avoid confusion.

\begin{proof} This follows from Lemma 2.1 of \cite{B_imrn}.  Namely, there exists a unique \lq\lq\,$G$-valued\rq\rq\   formal solution of the form
$\tPsiz_f(\zeta)=(I+\sum_{k\geq 1}\tilde\psi_k \zeta^k) \zeta^{\Lambda} e^{\frac{1}{\zeta}\tilde E_+}$.
A special feature of our situation is that the formal monodromy is trivial, i.e.\ $\Lambda=0$.  This follows from the algorithm given in the proof of Lemma 2.1 of \cite{B_imrn} (formulae (A.8) and (A.9) there), 
because the coefficients of $\zeta^{-2}$ and  $\zeta^{-1}$
in equation (\ref{merom})
belong to Cartan subalgebras which are in apposition
and are therefore orthogonal (by \cite{Ko59}, Theorem 6.7).
The required formal solution is then $\Psiz_f(\zeta)=\tPsiz_f(\zeta) e^{-w}$ (with $\psi_k=e^{w}\tilde\psi_k e^{-w}$).

In order to explain where the formula comes from, we sketch its derivation in Appendix \ref{sec:formalsolution}. In the notation of Appendix \ref{sec:formalsolution} we have
$A_{-1}=-\tilde E_+=\Ad(e^{-w})(-E_+)$,
$\Lambda_{-1}=-E_+$,  and $P=e^{-w}$.  We take $\fh_1=\fg^{-\tilde E_+}=\fg^{\tilde E_+}$ (the Cartan subalgebra containing $\tilde E_+$), and 
$\fh_2=\fg^{-E_+}=\fg^{E_+}$ (the Cartan subalgebra containing $E_+$). The triviality of the formal monodromy means that
$\Lambda_0=0$, where 
$\Lambda_0=\text{Proj}(\Ad(P^{-1})(A_0))$ and $\text{Proj}:\fg \to\fg^{E_+}$ is the projection map. To verify this, we observe that $A_0=-x w_x\in\fh$, and 
$\Lambda_0=\text{Proj}(\Ad(e^{-w})(-x w_x))=
\text{Proj}(-x w_x)$. This is zero because $\fh$ and $\fg^{E_+}$ are in apposition
and therefore orthogonal.
\end{proof}

From now on we shall need the Lie group automorphisms $\tilde\tau, \tilde\sigma, \tilde\chi, \tilde\theta$ induced by the Lie algebra automorphisms $\tau, \sigma, \chi, \theta$.  Recall (cf.\ page 49 of \cite{Kn02}) that a Lie algebra automorphism $f$ induces an automorphism $\tilde f$ of the (simply connected) Lie group $G$ satisfying $\tilde f\circ~\exp=\exp\circ~f$. Note
that $\tilde\tau $ is just conjugation by $P_0=e^{2\pi\i x_0/s}$.

\begin{pro}\label{psif-symm}
Under the same conditions as Proposition \ref{alphasymmetries}, the
formal solution $\Psiz_f$ satisfies

(i) $\tilde\tau (\Psiz_f(e^{2\pi\i/s} \zeta)) =\Psiz_f(\zeta)$

(ii) $\tilde\sigma(\Psiz_f(-\zeta))=\Psiz_f(\zeta) $


\noindent for all $\zeta\in\mathbb C^\ast$.
\end{pro} 

\begin{proof} In each case, it follows from Proposition \ref{hatbetasymmetries} that the left hand side satisfies (\ref{merom}). Substitution of the formula for $\Psiz_f(\zeta)$ from Lemma \ref{lm:formalsolution} yields a formal solution of the same type as $\Psiz_f(\zeta)$.  By uniqueness of the formal solution, this must coincide with $\Psiz_f(\zeta)$.
\end{proof} 

According to Definition 2.2 of \cite{B_imrn}, the singular directions (or anti-Stokes directions) associated to the connection form $\hat\beta$ at $\zeta=0$ are the rays in $\bC$ from the origin to the complex numbers
\begin{equation*}
\{
\beta(-\tilde{E}_+)
\st
\beta\in\Delta^{-\tilde{E}_+}
\}
\end{equation*}
where $\Delta^{-\tilde{E}_+}$ ($=\Delta^{\tilde{E}_+}$)
denotes the roots of $\fg$ 
with respect to the 
Cartan subalgebra $\fg^{-\tilde{E}_+}$ ($=\fg^{\tilde{E}_+}$).
This is just the set of nonzero eigenvalues of
$\ad(-\tilde E_+)$ (i.e., of $\ad(\tilde E_+)$).  Since
$\ad(\tilde E_+) = \Ad e^{-w}\circ \ad(E_+) \circ \Ad e^{w}$, the same set can be described as
\begin{equation*}
\{
\beta(-{E}_+)
\st
\beta\in\Delta^{-{E}_+}
\}
\end{equation*}
where $\Delta^{-{E}_+}$ ($=\Delta^{{E}_+}$)
denotes the roots of $\fg$ 
with respect to the 
Cartan subalgebra $\fg^{-{E}_+}$ ($=\fg^{{E}_+}$).
We shall use this description from now on.
To simplify notation we shall write
\[
\fh^\prime=\fg^{{E}_+},\quad \Delta^\prime=
\Delta^{{E}_+}.
\]

\begin{thm}\label{thm:singulardirection}
Assume that $l$ (the rank of $\fg$) is greater than $1$. Then
there are $2s$ equally spaced singular directions, where $s$ is the Coxeter number.
\end{thm}

\begin{proof} First we give a short direct proof which is valid for the classical simple Lie algebras. An alternative proof, valid for all simple $\fg$, will be given in Appendix \ref{sec:coxeterplane}. The direct proof exploits the fact that the adjoint representation
can be expressed in terms of a \lq\lq standard representation\rq\rq\ whose dimension is close to the Coxeter number. 

Inspired by \cite{MOP81}, we note that for any $N$-dimensional complex Lie group representation $\tilde\theta:G\to\Aut(V)$ and associated Lie algebra representation $\theta:\fg\to\End(V)$, the nonzero eigenvalues of $\theta(E_+)$ occur in orbits of the $s$-th roots of unity, each orbit having  exactly $s$ elements.  This follows from 
$
\vert \theta(E_+)-\lambda\vert =
\vert \tilde\theta(P_0) \theta(E_+) \tilde\theta(P_0)^{-1} -\lambda\vert =
\vert \theta(\Ad(P_0)E_+)-\lambda\vert =
\vert e^{2\pi\i/s} \theta(E_+) -\lambda\vert=
e^{2N\pi\i/s}\vert  \theta(E_+) -e^{-2\pi\i/s}\lambda\vert.
$

For each of $\fg=\fsl_{n+1}\bC$ $(n> 1)$, $\fso_n\bC$ $(n\ge 5)$,
$\fsp_n\bC$ $(n> 1)$, we take $\theta$ to be the standard representation (thus $N=n+1, n, 2n$ respectively). The Coxeter number $s$ satisfies $s\le N< 2s$ in each case (e.g.\ from Table 1 of \cite{MOP81}). Therefore there is exactly one orbit of nonzero eigenvalues, giving
$s$ equally spaced directions.

On the other hand it is well known that the adjoint representations are given,
respectively, by
$\theta^\ast\otimes\theta - 1, \wedge^2\theta, S^2\theta$. Since we are assuming that $l>1$ (hence $s>2$), it follows that the nonzero eigenvalues of $\ad E_+$ give exactly $2s$ equally spaced directions in the complex plane.
\end{proof}

Let us label the singular directions $d_1,\dots,d_{2s}$ clockwise.  For the moment we choose the initial direction $d_1$ arbitrarily; in the next section we shall see that there is a preferred choice. In view of the cyclic symmetry, it will be convenient to use singular directions with fractional indices $d_1,d_{1+\frac 1s},\dots,d_{2+\frac {s-1}s}$ in this section (although we shall revert to $d_1,\dots,d_{2s}$ in the next section). We denote the corresponding angles by
\[
\theta_{1+\frac js},\quad 0\le j\le 2s-1
\]
and then extend the notation by allowing $j\in\bZ$.  As in \cite{B_imrn} we
define sectors at the origin in $\bC$ by
$
\supersec_i =
(\theta_i -\tfrac\pi2,\theta_{i-\frac 1s} +\tfrac\pi2).
$
We have $\supersec_i =\supersec_{i+2}$.  In the universal covering $\tilde\bC^\ast$ we have corresponding sectors
$
\tildesupersec_i =
(\theta_i -\tfrac\pi2,\theta_{i-\frac 1s} +\tfrac\pi2),
$
and $\tildesupersec_i =\tildesupersec_{i+2}+2\pi$.

\begin{lm}\label{lm:canonicalsolution} 
There is a unique $G$-valued holomorphic \lq\lq canonical fundamental solution\rq\rq\  $\Psiz_i$ on $\supersec_i$ whose asymptotic expansion is $\Psiz_f$ as $\zeta\to 0$ in the sector $\supersec_i$. 
\end{lm}

\begin{proof}
This follows from Theorem 2.5 of \cite{B_imrn}. Namely 
there is a unique $G$-valued holomorphic solution
$\tPsiz_i$ on $\supersec_i$ whose asymptotic expansion is $\tPsiz_f$ as $\zeta\to 0$ in the sector $\supersec_i$. 
Here, $\tPsiz_f$ is the formal solution defined in the proof of Lemma \ref{lm:formalsolution}.
The required solution is then $\Psiz_i(\zeta)=\tPsiz_i(\zeta) e^{-w}$.

Let us point out here the relation with Stokes sectors and the classical asymptotic existence theorem, as explained in chapter 1 of \cite{FIKN06}.  First, a Stokes sector is a sector which contains exactly one from each pair of Stokes rays $\pm r_i$, where the Stokes rays (in the case of a pole of order $2$) are the rays orthogonal to the singular directions $d_i$. Then $\supersec_i $ is a (maximal) Stokes sector.
The existence theorem states that, on any Stokes sector,  there is a unique 
canonical fundamental solution whose asymptotic expansion is the chosen formal solution, as $\zeta\to 0$ in that sector.
\end{proof}

We may regard $\Psiz_i$ as defined on $\tildesupersec_i$ and then extend it by analytic continuation to
$\tilde\bC^\ast$. With this convention we have
\begin{equation}\label{analyticcontinuation}
\Psiz_{i+2}(e^{-2\pi\i}\zeta)=\Psiz_i(\zeta)
\end{equation}
because, on $\supersec_i$, both have the same asymptotic expansion $\Psiz_f$, and therefore must be equal. By standard abuse of notation, $e^{-2\pi\i}\zeta$ is to be interpreted in the universal covering.

Observe that 
\[
\tildesupersec_i\cap\tildesupersec_{i+\frac1s}=
\{z\in\tilde\bC^\ast
\st
\theta_i-\tfrac\pi2 < \arg z < \theta_i+\tfrac\pi2\}.
\]
This is a sector of width $\pi$ which is symmetrical about $\theta_i$. 
On this sector (and hence on all of $\tilde\bC^\ast$) the solutions
$\Psiz_{i+\frac1s},\Psiz_i$
must agree up to a constant multiplicative factor.
Thus, associated to the singular direction $\theta_i$ (more precisely, to
$\tildesupersec_i\cap\tildesupersec_{i+\frac1s}$),
we have the {\em Stokes factor} $\Kz_{i}\in G$, defined by  
\[
\Psiz_{i+\frac1s}(\zeta) =\Psiz_i(\zeta) \Kz_{i}.
\]  
The {\em Stokes matrix} $\Sz_{i}\in G$ is defined similarly by  $\Psiz_{i+1}(\zeta) =\Psiz_i(\zeta) \Sz_{i}$.  

\begin{pro}\label{KandS} For all $i\in \tfrac1s\bZ$ we have

(i) $\Sz_i=\Kz_i \Kz_{i+\frac1s}\dots \Kz_{i+\frac{s-1}s}$

(ii) $\Kz_{i+2}=\Kz_i$, $\Sz_{i+2}=\Sz_i$

(iii) $\Psiz_i(e^{2\pi\i}\zeta)=\Psiz_i(\zeta) \Sz_i\Sz_{i+1}$ (i.e.\ the monodromy of
$\Psiz_i$ is $\Sz_i\Sz_{i+1}$).

\end{pro}  

\begin{proof} (i) is obvious. (ii) and (iii) follow from (\ref{analyticcontinuation}) after substituting the definitions of the Stokes factors/matrices.
\end{proof}

Regarding (i), we note that $\Sz_i$  (more generally, any product of at most $s$ consecutive Stokes factors) determines the constituent Stokes factors $\Kz_i, \Kz_{i+\frac1s},\dots, \Kz_{i+\frac{s-1}s}$, by the remarks before Lemma 2.7 of \cite{B_imrn}.

\begin{pro}\label{psi-symm}
Under the same conditions as Proposition \ref{alphasymmetries}, the
canonical solutions $\Psiz_i$ satisfy

(i) $\tilde\tau (\Psiz_{i-\frac2s}(e^{2\pi\i/s} \zeta)) =\Psiz_i(\zeta)$

(ii) $\tilde\sigma(\Psiz_{i-1}(e^{\pi\i}\zeta))=\Psiz_i(\zeta) $


\noindent for all $\zeta\in\mathbb C^\ast$.
\end{pro} 

\begin{proof}  In each case, restricting to $\tildesupersec_i$, 
both sides are solutions,  and by Proposition \ref{psif-symm} they have
the same asymptotic expansion.  By the uniqueness
in Lemma \ref{lm:canonicalsolution}, they must coincide.
\end{proof} 

\begin{pro}\label{K-symm}
Under the same conditions as Proposition \ref{alphasymmetries}, the
Stokes factors $\Kz_i$ satisfy

(i) $\tilde\tau (\Kz_{i-\frac2s}) =\Kz_i$

(ii) $\tilde\sigma(\Kz_{i-1})=\Kz_i$


\noindent for all $\zeta\in\mathbb C^\ast$.
\end{pro} 

\begin{proof} These 
follow from (\ref{psi-symm}) after substituting the definition of Stokes factor.
\end{proof} 

Proposition \ref{psi-symm} (i)  and the definition of the $\Kz_{i}$
give
$\tilde\tau (\Psiz_{i}(e^{2\pi\i/s} \zeta)) =
\Psiz_{i}(\zeta)  \Kz_{i} \Kz_{i+\frac1s}$, in other words
\begin{equation}
P_0 \Psiz_{i}(e^{2\pi\i/s} \zeta) =
\Psiz_{i}(\zeta)  \Kz_{i} \Kz_{i+\frac1s} P_0.
\end{equation}
We regard this as a \lq\lq partial monodromy\rq\rq\  formula. 
Applying it (or Proposition \ref{K-symm} (i)) $s$ times we see that the group 
element $\Kz_{i} \Kz_{i+\frac1s} P_0$ is an $s$-th root of the monodromy. By
Proposition \ref{psif-symm} (i), $P_0$ plays the role of an $s$-th root of the formal monodromy. 

As $\Kz_{1} \Kz_{1+\frac1s} P_0$ determines the individual Stokes factors
$\Kz_{1},\Kz_{1+\frac1s}$ (by the remarks above) and hence generates all Stokes data, it will be our main focus. We introduce the following notation for it:

\begin{df}\label{def:qqpi}
$M^{(0)}=\Kz_{1} \Kz_{1+\frac1s} P_0$.
\end{df}

To investigate $M^{(0)}$ we need further results from \cite{B_imrn}.

\begin{lm}(\cite{B_imrn}, Definition 2.3 and Lemma 2.4)\label{lm:stokes}

(i) Let $\cR(d_i)$ denote the the set of roots with respect to the Cartan subalgebra $\fh^\prime$ supporting the singular direction $d_i$, i.e.\ 
the set of $\beta\in\Delta^\prime$
for which $\beta(-E_+)\in\bC^\ast$ has direction $d_i$. We have
$\Delta'=\cR(d_1)\sqcup\cR(d_{1+\frac{1}{s}})\sqcup\cdots\sqcup\cR(d_{2+\frac{s-1}{s}})$ (disjoint union).

(ii) The {\em group of Stokes factors associated to} $d_i$ is defined to be the group 
\[
\sto(d_i)=\overset{\scriptstyle }{\underset{\beta\in\cR(d_i)}\Pi} U_\beta
\]
where $U_\beta=\exp(\fg^\prime_\beta)\subset G$ 
and
$\fg^\prime_\beta=\{ \xi\in\fg \st [h,\xi]=\beta(\xi) \ \forall h\in\fh^\prime\}$. The product of the $U_\beta$ can be taken in any order.

(iii) $\Delta'_+=\cR(d_1)\sqcup\cR(d_{1+\frac{1}{s}})\sqcup\cdots\sqcup\cR(d_{1+\frac{s-1}{s}})$ is a set of positive roots of $\Delta^\prime$. The product of the groups $\sto(d_1),\dots,\sto(d_{1+\frac{s-1}{s}})$ is the unipotent part of the Borel subgroup defined by this set of positive roots.
\end{lm}

\begin{lm}(\cite{B_imrn}, Lemma 2.7)
$\Kz_i\in\sto(d_i)$.
\end{lm}

\begin{df}\label{def:cm} The subspace $\cM_{\Delta'_+}\subseteq G$ of \lq\lq abstract Stokes data\rq\rq\ for the tt*-Toda equations is defined to be
\[
\cM_{\Delta'_+}=
\sto(d_1) \sto(d_{1+\frac{1}{s}}) P_0=
\{
ABP_0 \in G \mid A\in \sto(d_1), B\in \sto(d_{1+\frac{1}{s}})
\}.
\]
\end{df}

Thus $M^{(0)} \in \cM_{\Delta'_+}$. In the next section we shall describe the space $\cM_{\Delta'_+}$ more explicitly, and in the following section we shall compute $M^{(0)}$ for some solutions of the tt*-Toda equations.

\section{Lie-theoretic description of the Stokes data}\label{lie}

From section \ref{data} we know that the Stokes factors
$\Kz_{1}\in\sto(d_1)$, 
 $\Kz_{1+\frac1s}\in\sto(d_{1+\frac{1}{s}})$ contain all the Stokes data of the meromorphic connection $\hat\beta$. Specifying $\Kz_{1},\Kz_{1+\frac1s}$ is equivalent to specifying the \lq\lq partial monodromy\rq\rq\ matrix
$M^{(0)}=\Kz_{1} \Kz_{1+\frac1s} P_0$ (Definition \ref{def:cm}). The Stokes group $\sto(d_i)$ is the unipotent subgroup of $G$ corresponding to the Lie algebra
$\oplus_{\beta\in\cR(d_i)}\,\fg^\prime_\beta$, where $\cR(d_i)$ is the set of roots supported by the singular direction $d_i$. 
 
The question now arises of computing $\cR(d_i)$ explicitly.  In sub-section \ref{lie1} we answer this question in the following way. 

First, by Lemma \ref{lm:stokes}, we know that
$\cR(d_1)\sqcup\cR(d_{1+\frac{1}{s}})\sqcup\cdots\sqcup\cR(d_{1+\frac{s-1}{s}})=\Delta^\prime_+$, a positive root system with respect to the Cartan subalgebra $\fh^\prime$. We shall show (Theorem \ref{thm:headandtail}) that 
$\cR(d_1)\sqcup\cR(d_{1+\frac{s-1}{s}})$ is nothing but the associated set
$\Pi^\prime$ of simple roots.  Thus the sector consisting of a \lq\lq half period\rq\rq\  of singular directions gives the positive system, and its first and last singular directions --- the head and tail of the sector --- give the simple roots.

The resulting partition of $\Pi^\prime$ into two disjoint subsets was known to Kostant and Steinberg in a purely Lie-theoretic context. We use this to deduce that
$\cR(d_1)\sqcup\cR(d_{1+\frac{1}{s}})$ consists of those positive roots $\beta\in\Delta^\prime_+$ which become negative under the action of
$\gamma^{-1}$ (and we show that $\cR(d_i)$ can be expressed in a similar way for any $i$). This gives a Lie-theoretic characterization of the possible values of the Stokes factors $\Kz_{1},\Kz_{1+\frac1s}$ of the meromorphic connection.

In sub-section \ref{lie2} we cast this Stokes data into a canonical form, which will be used (in the next section) to find the actual Stokes factors associated to some particular solutions of the tt*-Toda equations.  

It may be worth emphasizing here that, in general, it is difficult to compute the Stokes data of an o.d.e., as it is a nontrivial refinement of the \lq\lq visible\rq\rq\ data such as the monodromy matrix. The standard method depends on knowing an integral formula for a specific solution, which can be expanded asymptotically in different directions. Even the computation of the monodromy matrix at a regular singular point can present difficulties, because of the possibility of resonance.  And even if these obstacles can be surmounted, there is in general no canonical way to present the Stokes data, as its construction depends on making several choices. 

However, in our situation, the above special structure of $\sto(d_i)$ allows us to show that $M^{(0)} \in \cM_{\Delta'_+}$ can be computed exactly from its conjugacy class. The conjugacy classes which occur can be described in a canonical fashion, even when $M^{(0)} \in \cM_{\Delta'_+}$ fails to be semisimple (which is a manifestation of resonance). 
Thus we are in a very favorable situation.

From now on in this section we revert to the notation
$d_1,\dots,d_{2s}$ for the singular directions (rather than the fractional index notation
$d_1,\dots,d_{2+\frac{s-1}{s}}$ of the previous section).

\subsection{Head and tail}\label{lie1}
$ $

Let $G$ be a complex simple Lie group of rank $l>1$, with Lie algebra $\fg$. Let $\fh$ be a Cartan subalgebra, with associated Weyl group $W$.
Let $\Delta_+$ be a
system of positive roots, with associated 
simple roots $\Pi=\{\alpha_1,\dots,\alpha_l\}$.

Denote by $(\cdot,\cdot)$ the bilinear form on $\fg^\ast$ corresponding to $B(\cdot,\cdot)$ on $\fg$ via the identification $f(\cdot)\leftrightarrow
B(\cdot,H_f)$. If $\alpha,\beta$ are roots and $R_\alpha$, $R_\beta\in W$ are the corresponding reflections, we have $R_{\alpha}(\beta)=\beta-2\tfrac{(\beta,\alpha)}{(\alpha,\alpha)}\alpha$.

\begin{pro}\cite{Steinberg}\label{pro:partition}
There is a partition $\{1,2,\dots,l\}=I_1\sqcup I_2$ such that all elements of $\Pi_1=\{ \alpha_i \in \Pi \st i\in I_1 \}$ are orthogonal, and all elements of $\Pi_2=\{ \alpha_i \in \Pi \st i\in I_2 \}$ are orthogonal. This partition is unique up to the labelling of $I_1,I_2$. 
\end{pro}

Recall (e.g.\ from \cite{H}) that the product (in any order) of the $l$ simple root reflections  is called a Coxeter element (of $W$).  
We consider the Coxeter element 
\[
\textstyle
\gamma=\tau_2\tau_1,\quad
\tau_k=\prod_{\alpha_i\in \Pi_k} R_{\alpha_i}
\]
By orthogonality, the order of the products inside $\tau_1$ or $\tau_2$ does not matter.  For this reason, we have $\tau_1^2=\tau_2^2=1$.

For any $t\in W$, let 
\[
\Lambda(t)=t^{-1}\Delta_-\cap\Delta_+
\]
i.e.\  the set consisting of those positive roots which become negative upon applying $t$. It is known (\cite{Kostant85},\cite{H}) that the cardinality of $\Lambda(t)$
is the length of $t$ in $W$, which is $l$ when $t$ is a Coxeter element.

For $i=1,2,3,\dots$ let
\[
\tau^{(n)}=\tau_n\tau_{n-1}\cdots\tau_1,
\quad
\tau_i=
\begin{cases}
\tau_1\ \text{if $i$ is odd}
\\
\tau_2\ \text{if $i$ is even}
\end{cases}
\]
Let $\tau^{(-n)}=(\tau^{(n)})^{-1}$ and $\tau^{(0)}=1$.
Kostant (\cite{Kostant85}, Proposition 6.8) showed that
\begin{equation}\label{eq:union}
\Lambda(\tau^{(n)})=\tau^{(0)}\Pi_1\sqcup\cdots\sqcup
\tau^{(-(n-1))}\Pi_n,
\quad
\Pi_i=
\begin{cases}
\Pi_1\ \text{if $i$ is odd}
\\
\Pi_2\ \text{if $i$ is even}
\end{cases}
\end{equation} 
In \cite{Kostant85}, Kostant has an overall assumption that $G$ has type A-D-E and even Coxeter number, however, the proof of formula (\ref{eq:union}) does not need these assumptions.

This gives us a characterization of the partition $\Pi=\Pi_1\sqcup\Pi_2$:

\begin{pro}\label{pro:Piwithgamma}  

$ $

(i) $\Lambda(\gamma)=\Pi_1\sqcup \gamma^{-1} (-\Pi_2)$, $\Lambda(\gamma^{-1})=\gamma(-\Pi_1)\sqcup \Pi_2$

(ii) $\Pi_1=\Pi\cap \Lambda(\gamma)$, $\Pi_2 =\Pi\cap \Lambda(\gamma^{-1})$
\end{pro}

\begin{proof}  (i) By formula (\ref{eq:union}) and the fact that $\tau_2(\Pi_2)=-\Pi_2$, we have 
$\Lambda(\gamma)=
\Lambda(\tau_2\tau_1)=
\tau^{(0)}\Pi_1\sqcup\tau^{(-1)}\Pi_2=
\Pi_1\sqcup(\tau_1)^{-1}\Pi_2=
\Pi_1\sqcup(\tau_1)^{-1}(\tau_2)^{-1}(-\Pi_2)=
\Pi_1\sqcup \gamma^{-1} (-\Pi_2)$.
Next, we have $\Lambda(\gamma^{-1})=
\gamma\Delta_-\cap\Delta_+=
\gamma(\Delta_-\cap\gamma^{-1}\Delta_+)=
\gamma(-\Delta_+\cap\gamma^{-1}(-\Delta_-))=
\gamma(-\Lambda(\gamma))$,
and this is $\gamma(-\Pi_1)\sqcup \Pi_2$ by the previous calculation.
(ii) It follows that $\Pi_1=\Pi\cap \Lambda(\gamma)$ and $\Pi_2 =\Pi\cap \Lambda(\gamma^{-1})$. 
\end{proof}

As we have seen in section \ref{eqns}, Kostant showed in 
\cite{Ko59} that the Cartan subalgebra $\fg^{z_0}$ is in apposition to the Cartan subalgebra $\fh$ with respect to the principal element $P_0$.  Corollary 8.6 of \cite{Ko59} then shows that $P_0$ is a group representative of a Coxeter transformation of $\fg^{E_+}=\fh^\prime$. Let us write
$A_\delta=P_0$, and
\[
\delta=\Ad P_0\vert_{\fh^\prime}\in W^\prime
\]
for the corresponding Coxeter element.

\begin{pro}\label{rotation}
For all $\beta\in\Delta^\prime$ we have
$\delta(\beta)(-E_+)=e^{\frac{2\pi\sqrt{-1}}{-s}}\beta(-E_+)\in\bC^*$.
\end{pro}

\begin{proof} Since $\Ad(A_\delta)(e_\beta)=e_{\delta(\beta)}$, we have
$[y,e_{\delta(\beta)}] = (\delta(\beta)(y)) e_{\delta(\beta)}
= (\delta(\beta)(y)) \Ad(A_\delta)(e_\beta)$ for all $y\in\fh^\prime$.
On the other hand
$[y,e_{\delta(\beta)}] =  [y,\Ad(A_\delta)(e_\beta)] = 
\Ad(A_\delta)[ \Ad(A_\delta^{-1})(y),e_\beta] =
\beta( \Ad(A_\delta^{-1})(y) ) \Ad(A_\delta)(e_\beta)$.
Thus 
$\delta(\beta)(y)=
\beta( \Ad(A_\delta^{-1})(y) )$ for all $y\in\fh^\prime$.
Since $A_\delta=P_0$, we have 
$\delta(\beta)(-E_+)=\beta(\Ad(P_0^{-1})(-E_+))
=e^{\frac{2\pi\sqrt{-1}}{-s}}\beta(-E_+)$ where we have used 
$\Ad(P_0)E_+=e^{\frac{2\pi\sqrt{-1}}{s}}E_+$.
\end{proof}

\begin{pro}\label{pro:P0asAgamma}
There is a choice of simple roots $\Pi'$ and partition $\Pi'=\Pi'_1\sqcup\Pi'_2$ 
(with respect to the Cartan subalgebra $\fh^\prime=\fg^{-E_+}$)
such that
$\delta=\tau_2\tau_1$, i.e.\ $\delta$ can be realized as the particular type of Coxeter element $\gamma=\tau_2\tau_1$ for $\Pi'=\Pi'_1\sqcup\Pi'_2$.
\end{pro}

\begin{proof} We shall construct $\Pi^\prime$ from the 
initial set of simple roots $\Pi=\{\alpha_1,\dots,\alpha_l\}$ (with
respect to $\fh$).  First we choose a permutation $\sigma$ 
so that (for some $k$)
$\alpha_{\sigma(1)},\dots,\alpha_{\sigma(k)}\in\Pi_2$ and $\alpha_{\sigma(k+1)},\dots,\alpha_{\sigma(l)}\in\Pi_1$, where
$\Pi=\Pi_1\sqcup\Pi_2$ is given by Proposition \ref{pro:partition}.  

According to Theorem 8.6 of \cite{Ko59}, one can choose $A\in G$ such that $\Ad(A)(\fh')=\fh$ and 
$\beta_1,\dots,\beta_l$
is a set of simple roots with respect to $\fh^\prime$,
where $\beta_i=\alpha_{\sigma(i)}\circ \Ad(A)$.
Then
$\epsilon_j^\prime=\Ad(A^{-1})(\epsilon_{\sigma(j)})$ 
satisfies
$\beta(\epsilon_j^\prime)=\delta_{ij}$, so
$\epsilon_1^\prime,\dots,\epsilon_l^\prime$
is the dual basis to $\beta_1,\dots,\beta_l$, and we have
$P_0=\exp(\frac{2\pi\sqrt{-1}}{s}\Ad(A)(x'_0))$
where $x'_0=\sum_{j=1}^l \epsilon'_j$.

On the other hand,  $(\beta_i,\beta_j)=(\alpha_{\sigma(i)}\circ \Ad(A),\alpha_{\sigma(j)}\circ \Ad(A))=(\alpha_{\sigma(i)},\alpha_{\sigma(j)})$, so $\beta_1,\dots,\beta_k$ are mutually orthogonal, and similarly $\beta_{k+1},\dots,\beta_l$ are mutually orthogonal.  Since the partition of $\Pi'$ with this orthogonality property is unique, we must have $\Pi'_2=\{\beta_1,\dots,\beta_k\}$ and $\Pi'_1=\{\beta_{k+1},\dots,\beta_{l}\}$.
Thus $\delta=R_{\beta_1}\dots R_{\beta_l}$ is of the form $\gamma=\tau_2\tau_1$ where $\tau_2$ is the product of the reflections in $\Pi_2'$ and $\tau_1$ is the product of the reflections in $\Pi_1'$.
\end{proof}

Now we apply this theory to the Stokes data of section \ref{data}, and in particular to the diagram of singular directions $d_1,\dots,d_{2s}$ (see Figure \ref{singulardirection} for an illustration when $s=6$).  
\begin{figure}[h]
\scalebox{0.4}{\includegraphics{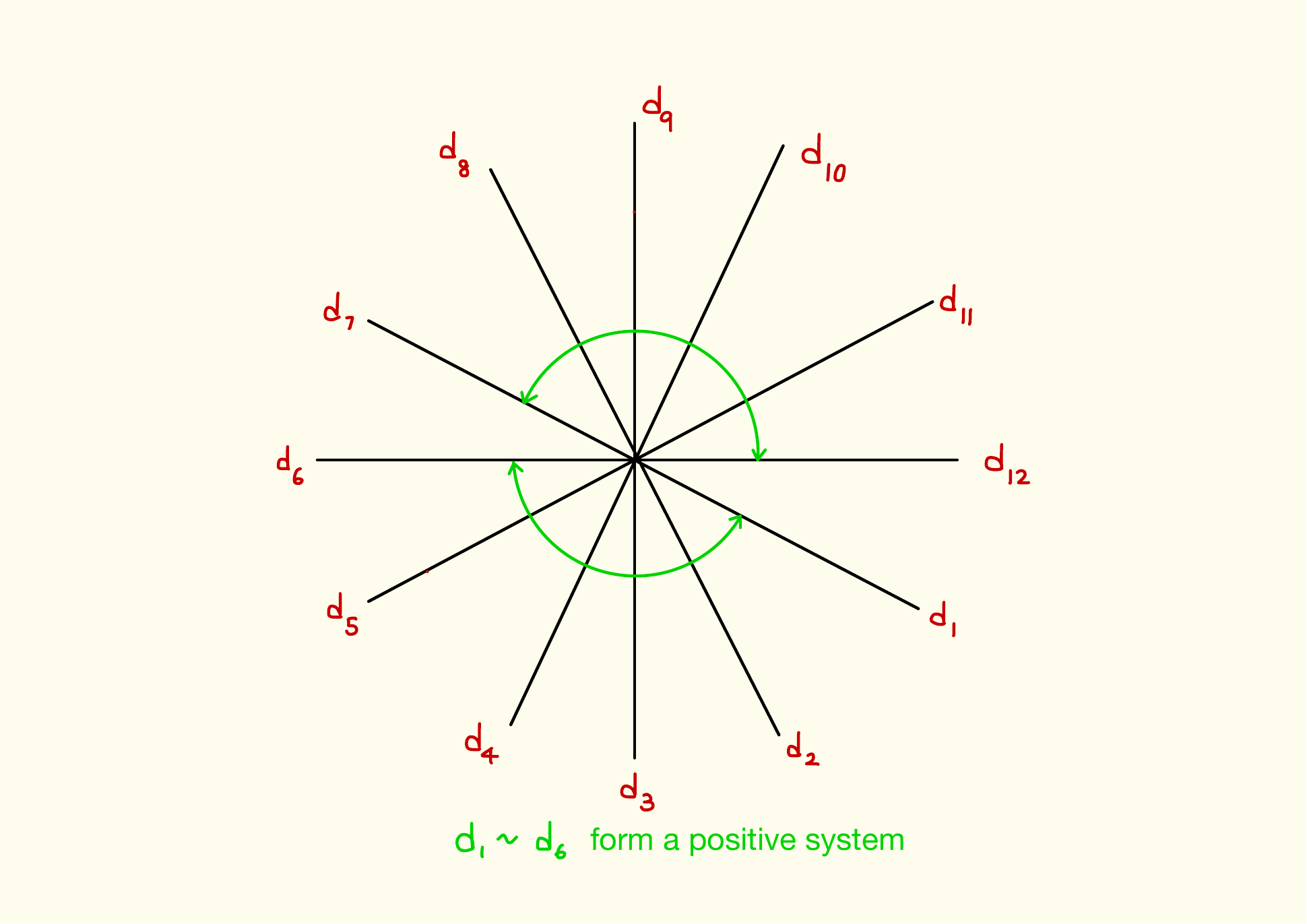}}
\caption{Singular directions  and positive system of roots}\label{singulardirection}
\end{figure}
Recall that the singular direction $d_i$ is the ray in $\bC$ at the origin which passes through the (nonzero) points $\beta(-{E}_+)$, where $\beta\in
\cR(d_i)$.

\begin{cor}\label{cor:skiponeray}
The action of $\delta$ (on the set of roots $\Delta^\prime$) moves the roots on the singular direction $d_i$ to the roots on the singular direction $d_{i+2}$ (where $i+2$ is interpreted mod $2s$).
\end{cor}

\begin{proof} By Theorem \ref{thm:singulardirection}, there are $2s$ singular directions, with successive directions separated by $\pi/s$. By Proposition \ref{rotation}, (the roots on) these singular directions are
rotated clockwise by $2\pi/s$. The statement follows.
\end{proof}

Up to this point we have allowed arbitrary choices of initial singular direction $d_1$. We shall now choose $d_1$ (and $A$) so that 
\[
\Delta'_+=\cR(d_1)\sqcup\cR(d_2)\sqcup\cdots\sqcup\cR(d_s).
\]
This is always possible because of the freedom in the choice of
$A$ in the proof of Proposition \ref{pro:P0asAgamma}. From now on we refer to the sequence $d_1,\dots,d_s$ of singular directions as the positive sector. 

Proposition \ref{pro:Piwithgamma} and Corollary \ref{cor:skiponeray} immediately imply:

\begin{cor}\label{cor:twodirections} $ $

(i) $\cR(d_{1})\sqcup\cR(d_{2})= 
\Lambda(\gamma^{-1})=
\Pi'_2 \sqcup \gamma(-\Pi'_1)$, and

(ii) $\cR(d_{s-1})\sqcup\cR(d_{s})=
\Lambda(\gamma)=
\Pi'_1 \sqcup \gamma^{-1}(-\Pi'_2)$
\end{cor}

We shall now refine this to show that the roots on  $\cR(d_{1})$ 
and $\cR(d_{s})$ --- the
head and tail of the positive sector --- are precisely the simple roots:

\begin{thm}\label{thm:headandtail}
$\cR(d_1)=\Pi'_2$, $\cR(d_2)=\gamma(-\Pi'_1)$,
$\cR(d_{s-1})=\gamma^{-1}(-\Pi'_2)$, $\cR(d_s)=\Pi'_1$.
\end{thm}

For this we need more ingredients. The first is:

\begin{lm}\label{lm:conj2}
Let $\{\alpha_i \st i\in I\}$ be a nonempty set of mutually orthogonal simple roots. 
If all $c_i\in \bZ_{\geq 0}$ and at least two $c_i$ are nonzero, then
$\sum_{i\in I}c_i\alpha_i$ cannot be a root.  
\end{lm}

\begin{proof} 

Let $\{\alpha_i^\prime \st i\in I^\prime\}$
be a set of positive (but not nesssarily simple) roots that are mutually orthogonal. Suppose that $\sum_{i\in I'}c'_i \alpha'_i$ is a root, where all $c'_i\in\bZ_{>0}$. Then we claim that any combination $\sum_{i\in I'}c''_i\alpha'_i$ with $0\leq c''_i\leq c'_i$ is also a root (or zero). 

To prove this we use the fact that, if $\alpha,\beta$ are any two roots with $(\alpha,\beta)>0$, then $\alpha-\beta$ is either zero or a root. Namely, if $\sum_{i\in I'}c'_i\alpha'_i$ is a root, and
$c'_j>0$, then $(\sum_{i\in I'}c'_i\alpha'_i,\alpha'_j)=c'_j|\alpha'_j|^2>0$
and $\sum_{i\in I'}c'_i\alpha'_i-\alpha'_j$ is again a root. Repeating the process we can obtain any linear combination $\sum_{i\in I'}c''_i\alpha'_i$.

Let us assume now that $\sum_{i\in I}c_i\alpha_i$ is a root, with all $c_i\in \bZ_{\geq 0}$ and at least two $c_i$ nonzero. By the assertion just proved, we obtain at least one root of the form 
$\alpha_i+c\alpha_j$ with $i,j\in I$ and $c\in \bZ_{> 0}$.
However, this contradicts the well known properties of root strings (see e.g.\ \cite{Kn02}): if $\alpha$ and $\beta$ are simple roots, the roots of the form
$\beta+c\alpha$, $c\in\bZ$, are exactly those given by 
$c=-p,-p+1,\dots,q$ for some $p,q\in \bZ_{\geq 0}$
with  $p-q=\frac{2(\alpha,\beta)}{(\alpha,\alpha)}$.  This is not possible because $\alpha_i$ and $\alpha_i+c\alpha_j$ belongs to the root string, yet $p-q=0$.
\end{proof}

\begin{lm}\label{lm:nononsimple} We have
$\cR(d_1)\subseteq \Pi_2^\prime$ and 
$\cR(d_s)\subseteq \Pi_1^\prime$.
In particular, the head and tail of the positive sector contain only simple roots.
\end{lm}

\begin{proof} Let $\Pi^\prime=\{\beta_1,\dots,\beta_l\}$.
Let $\beta=\sum_{i=1}^l c_i\beta_i$, $c_i\in\bZ_{\geq 0}\in\Delta^\prime$. If $\beta(-E_+)$ falls on the ray $d_s$, then plane geometry implies that $\beta_j(-E_+)$ also falls on the ray $d_s$ whenever $c_j\neq 0$.  
Then $\beta_j \in \Pi^\prime\cap \cR(d_s)
\subseteq \Pi^\prime\cap (\cR(d_s)\cup \cR(d_{s-1}))
= \Pi^\prime\cap \Lambda(\gamma)$ and this is $\Pi_1^\prime$
by Proposition \ref{pro:Piwithgamma}.
As the roots in $\Pi'_1$ commute,  
this contradicts Lemma \ref{lm:conj2} unless 
$\beta=\beta_j$ for some $j$. Thus $\beta\in\Pi_1^\prime$.
Similarly, if  $\beta$ is a positive root such that $\beta(-E_+)$ falls on the ray $d_1$, then $\beta\in\Pi_2^\prime$.
\end{proof}

Now we can give the proof Theorem \ref{thm:headandtail}  

\begin{proof}[Proof of Theorem \ref{thm:headandtail}]
It suffices to prove that
$\cR(d_1)=\Pi'_2$ and $\cR(d_s)=\Pi'_1$, as the remaining formulae follow from this and
Corollary \ref{cor:twodirections}.

We shall denote by $\Card X$ the cardinality of a (finite) set $X$.
As roots come in pairs, $\Card\Delta'=sl$ must be even. We consider separately the cases (1) $s$ even ($l$ odd or even), (2) $s$ odd ($l$ even).

\noindent Case 1: $s$ is even.

By Theorem \ref{thm:singulardirection} there are $2s$ singular directions, hence any two consecutive singular directions account for $l$ roots. In particular
\[
\Card\cR(d_1)+\Card\cR(d_2)=l,\quad \Card\cR(d_1)\cdot\tfrac{s}{2}+\Card\cR(d_2)\cdot\tfrac{s}{2}=\tfrac{sl}{2},
\]
the last expression being the number of roots in the positive sector.
Moreover,  $\Card\cR(d_s)=\Card\cR(d_2)$ because the Coxeter element $\gamma$ move the roots supported by $d_i$ to those supported by $d_{i+2}$.
By Lemma \ref{lm:nononsimple} we have
\begin{gather*}
\cR(d_s)\subseteq \Pi'_1\Rightarrow \Card\cR(d_s)\leq \Card\Pi'_1
\\
\cR(d_1)\subseteq\Pi'_2\Rightarrow \Card\cR(d_1)\leq \Card\Pi'_2,
\end{gather*}
which gives
\begin{align*}
l=\Card\Pi'_1+\Card\Pi'_2&\geq \Card\cR(d_s)+\Card\cR(d_1)
\\
&=\Card\cR(d_2)+\Card\cR(d_1)=l.
\end{align*}
Thus
$\Card\Pi'_1=\Card\cR(d_s)=\Card\cR(d_2)$, $\Card\Pi'_2=\Card\cR(d_1)$.  It follows that
$\Pi'_1=\cR(d_s)$ and $\Pi'_2=\cR(d_1)$, as required.

\noindent Case 2: $s$ is odd.

As in Case 1, we have
\[
\Card\cR(d_1)+\Card\cR(d_2)=l,\quad \Card\cR(d_1)\cdot\tfrac{s+1}{2}+\Card\cR(d_2)\cdot\tfrac{s-1}{2}=\tfrac{sl}{2},
\]
which implies $\Card\cR(d_1)=\Card\cR(d_2)$.
Moreover,  $\Card\cR(d_s)=\Card\cR(d_1)$ because 
$\gamma^{\frac12(s+1)}\Card\cR(d_1)=\Card\cR(d_s)$.

By Lemma \ref{lm:nononsimple} we have
\begin{gather*}
\cR(d_s)\subset \Pi'_1\Rightarrow \Card\cR(d_s)\leq \Card\Pi'_1 
\\
\cR(d_1)\subset\Pi'_2\Rightarrow \Card\cR(d_1)\leq \Card\Pi'_2,
\end{gather*}
which implies that
\begin{align*}
l&=\Card\Pi'_1+\Card\Pi'_2\geq \Card\cR(d_s)+\Card\cR(d_1)
\\
&=\Card\cR(d_1)+\Card\cR(d_1)=\Card\cR(d_2)+\Card\cR(d_1)=l.
\end{align*}
Thus
$\Card\Pi'_1=\Card\cR(d_s)=\Card\cR(d_1)=\Card\cR(d_2)$, $\Card\Pi'_2=\Card\cR(d_1)$.

\noindent Thus,
$\Pi'_1=\cR(d_s)$ and $\Pi'_2=\cR(d_1)$, as required. In this case
$\Card\cR(d_i)=l/2$
for all $i$.
\end{proof}

\begin{cor}
The roots associated to any singular direction are mutually orthogonal, and are given by
\begin{eqnarray*}
&& \cR(d_1)=\Pi'_2,~\cR(d_3)=\gamma(\Pi'_2),~\dots,~\cR(d_{2s-1})=\gamma^{s-1}(\Pi'_2),\\
&& \cR(d_2)=\gamma(-\Pi'_1),~\cR(d_4)=\gamma^2(-\Pi'_1),~\dots,~\cR(d_{2s})=\gamma^s(-\Pi'_1)=-\Pi'_1.
\end{eqnarray*}
\end{cor}


\begin{figure}[h]
\scalebox{0.4}{\includegraphics{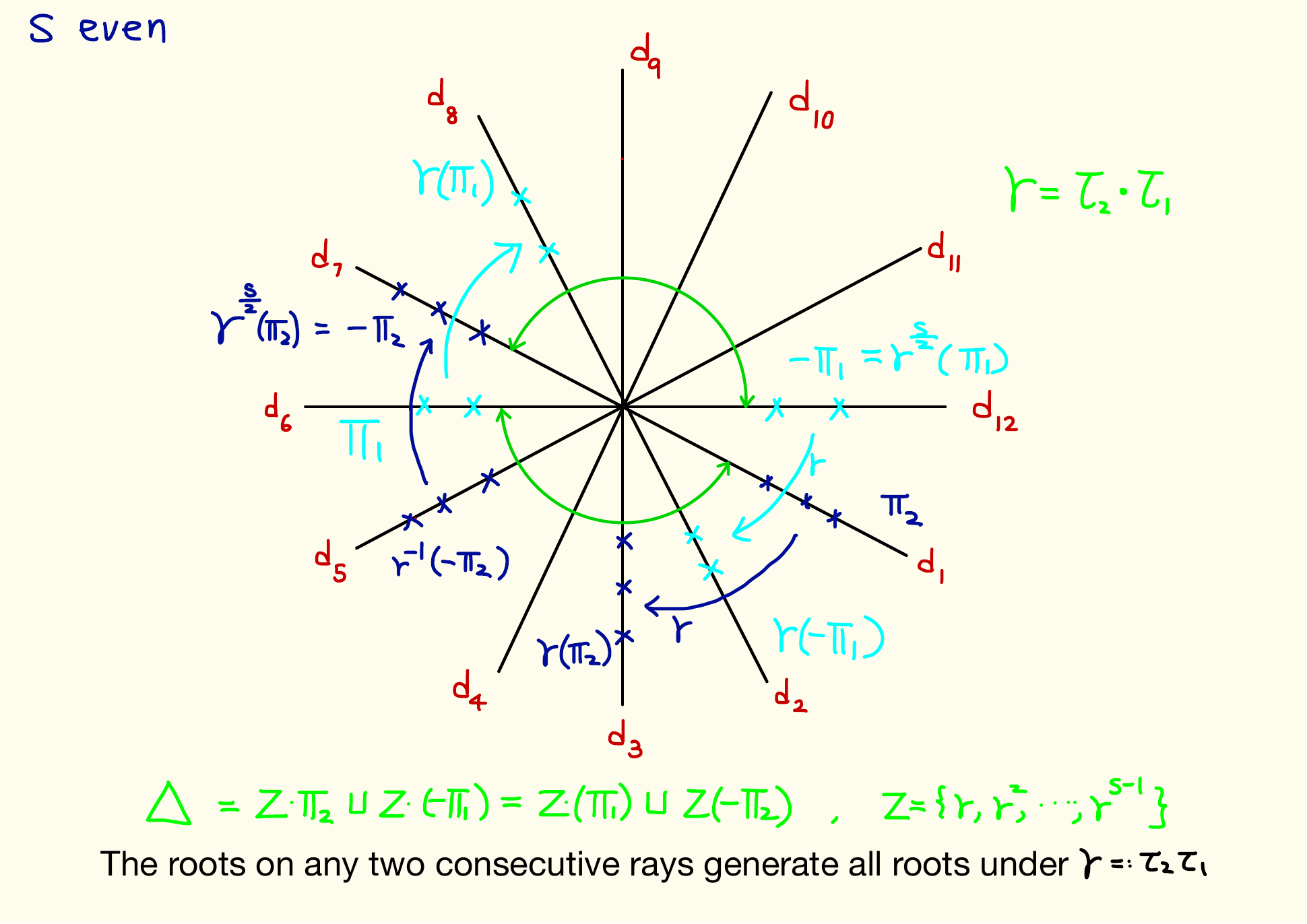}}
\caption{Simple roots and action of Coxeter element}\label{Roots}
\end{figure}

\begin{proof} The expressions for the $\cR(d_i)$ follow immediately from Corollary \ref{cor:skiponeray}  and Theorem \ref{thm:headandtail}.  
The orthogonality statement follows from the fact that $(\cdot,\cdot)$ is invariant under the reflections $R_\beta$ and hence under the action of $\gamma$. 
\end{proof}

\subsection{Steinberg cross-section}\label{lie2}
$ $

Recall (e.g.\ \cite{S}) that an element of $G$ is said to be regular if its centralizer in $G$ has dimension $l$.  We denote by $G^{reg}$ the subspace of $G$ consisting of regular elements.  
To study the conjugacy classes of $G$, Steinberg introduced the map
\[
\chi: G \to \bC^l, \quad g\mapsto (\chi_1,\dots,\chi_l)
\]
where the $\chi_i$ are the characters of the basic irreducible representations.

\begin{pro}\cite{S}(Theorem 4, page 120)\label{pro:S2}
Assume that $G$ is simply connected.
For any choice  of an ordered set of simple roots $\simple=(\alpha_1,\dots,\alpha_l)$, the map
\[
C^\simple:  (t_1,\dots,t_l)\in \bC^l \mapsto E_1(t_1)n_1\dots E_l(t_l)n_l\in G
\]
is a cross-section of $\chi|_{G^{reg}}$, where
$E_i(t_i)=\exp(t_i e_{\alpha_i})\in U_{\alpha_i}$ and the $n_i$ are group representatives of the Weyl group elements $R_{\alpha_i}$.
\end{pro}

Thus $\bC^l$  parametrizes the regular conjugacy classes, and $C^\simple$ gives a particular choice of representatives. We refer to $C^\simple$ as a {\em Steinberg cross-section.}

\begin{df}\label{def:rpi}
$
\cR^\simple=\{  E_1(t_1)n_1\dots E_l(t_l)n_l  \ \vert\ (t_1,\dots,t_l)\in \bC^l \}
$
\end{df}

In the previous sub-section we have chosen a system of positive roots $\Delta_+^\prime$, and a partition $\Pi'=\Pi'_1\sqcup\Pi'_2$ of the corresponding simple roots, with respect to which the space of
\lq\lq abstract Stokes data\rq\rq\ for the tt*-Toda equations has the following convenient description:
\[
\cM_{\Delta'_+}=
\sto(d_1) \sto(d_2) P_0
\]
where
\[
\sto(d_1)=\overset{\scriptstyle }{\underset{\beta\in\Pi_2^\prime}\Pi} U_\beta,
\quad
\sto(d_2)=\overset{\scriptstyle }{\underset{\beta\in
\gamma(-\Pi_1^\prime)
}\Pi} U_\beta.
\]

The main result of this section is:

\begin{thm}\label{thm:crosssection}
Let $\simple=(\beta_1,\dots,\beta_l)$, where $\beta_1,\dots,\beta_k$ is any ordering of $\Pi'_2$, and $\beta_{k+1},\dots,\beta_l$ is any ordering of $\Pi'_1$. Then
$\cM_{\Delta'_+}=\cR^\simple.$
\end{thm}

\begin{proof} First we note that 
\[
n_i\exp(te_\alpha)n_i^{-1}=\exp(\Ad(n_i)te_\alpha)=
\exp (te_\beta)
\]
where $\beta=R_{\alpha_i}(\alpha)=\alpha-2\tfrac{(\alpha,\alpha_i)}{(\alpha_i,\alpha_i)}\alpha_i$.
Since the roots in each $\Pi'_i$ are mutually orthogonal, we see immediately that
\begin{eqnarray*}
E_1(t_1)n_1\cdots E_k(t_k)n_k&=&E_1(t_1)\cdots E_{k}(t_k)~n_1\cdots n_k, \\
E_{k+1}(t_{k+1})n_{k+1}\cdots E_l(t_l)n_l &=& E_{k+1}(t_{k+1})\cdots E_{l}(t_{l})~ n_{k+1}\cdots n_l.
\end{eqnarray*}
Next we wish to evaluate
\[
n_1\cdots n_k \ 
\Pi_{i=1}^{l-k} E_{k+i}(t_{k+i})
=n_1\cdots n_k 
\left(
\Pi_{i=1}^{l-k} E_{k+i}(t_{k+i})
\right)
(n_1\cdots n_k)^{-1}(n_1\cdots n_k).
\]
Using the orthogonality property again we obtain, for $i=1,\dots,l-k$,
\begin{align*}
 R_{\beta_1}\cdots R_{\beta_k}(\beta_{k+i})&=R_{\beta_1}\cdots R_{\beta_{k+i-1}} (\beta_{k+i})
 \\
 &=R_{\beta_1}\cdots R_{\beta_{k+i}} (-\beta_{k+i})
 \\
 &=R_{\beta_1}\cdots R_{\beta_l} (-\beta_{k+i})=\gamma(-\beta_{k+i}),
\end{align*}
where $\gamma=\tau_2\tau_1=(R_{\beta_1}\cdots R_{\beta_k})(R_{\beta_{k+1}}\cdots R_{\beta_l})$.
It follows that
\[
n_1\cdots n_k 
E_{k+i}(t_{k+i})
(n_1\cdots n_k)^{-1}
=
\exp( t_{k+i} e_{
\gamma(-\beta_{k+i})
}).
\]
Hence
\begin{align*}
C^\simple(t_1,\dots,t_l)&=
E_1(t_1)\cdots E_{k}(t_k)~n_1\cdots n_k ~E_{k+1}(t_{k+1})\cdots E_{l}(t_{l})~ n_{k+1}\cdots n_l 
\\
&= E_1(t_1)\cdots E_{k}(t_k)
\exp ( t_{k+1} e_{\gamma(-\beta_{k+1})})
\cdots 
\exp ( t_l e_{\gamma(-\beta_{l})})
~A_\gamma
\end{align*}
where $A_\gamma=~n_1\cdots n_l$ is a group representative of the Coxeter element $\gamma$.  By Proposition \ref{pro:P0asAgamma}, we know that we can choose $n_i$ so that $n_1\cdots n_k n_{k+1}\cdots n_l=P_0$.  Hence
$C^\simple(t_1,\dots,t_l)$
can be expressed as a general element of
$\cM_{\Delta^\prime_+}$.
This completes the proof.
\end{proof}

\section{Stokes data for local solutions near zero}\label{local}

In sections \ref{data} and \ref{lie} we have described Lie-theoretically the space $\cM_{\Delta^\prime_+}$ of
\lq\lq abstract Stokes data\rq\rq\ for the tt*-Toda equations.  Any (local) solution $w$ of the tt*-Toda equations (\ref{ttt}) gives an isomonodromic deformation of the meromorphic connection form $\hb$ of (\ref{betahat}),  and we have shown that the Stokes data corresponding to such a solution can be described by a single Lie group element $M^{(0)}\in\cM_{\Delta^\prime_+}$.  From now on we denote the Lie group element corresponding to $w$ by $M^{(0)}_{\,w}$.

In this section (Theorem \ref{explicitQQPi}) we compute $M^{(0)}_{\,w}$  for any solution $w$ which is defined on a punctured disk of the form
$\{ z\in\mathbb C \st 0 < \vert z\vert < \epsilon_w \}$ and which has a logarithmic singularity at zero.  We refer to such solutions as \lq\lq local solutions near zero\rq\rq.  We shall compute $M^{(0)}_{\,w}$ in terms of the asymptotics of $w$ at zero, which is the data naturally associated to $w$.

In addition to this, we shall show (Theorem \ref{convexset}) that such 
$M^{(0)}_{\,w}$ correspond naturally to the points of a convex polytope, essentially the fundamental Weyl alcove of the compact real form $G_{\text{cpt}}$ of $G$.  This is of some interest, because $G_{\text{cpt}}$ does not appear in the description of the tt*-Toda equations (or the flat connection forms $\alpha,\hat\alpha$, which have noncompact monodromy groups). It is a special property of the particular type of solutions that we are considering.

We begin by sketching the construction of local solutions near zero which was given in section 2 of \cite{GIL3} for the case $G=SL_n\bC$, but this time for any complex simple simply-connected Lie group $G$. To simplify notation, we set all the numbers $c^\pm_i$ and all the coefficients $k_i$ in the tt*-Toda equation (\ref{ttt}) equal to $1$ from now on --- this suffices for the discussion of solutions, as the case of general $k_i>0$ reduces to this case on rescaling $w$.

First we consider a new connection form
\begin{equation}\label{omega}
\omega =\tfrac1\lambda \eta_- dz,
\quad
\eta_-=\sum_{i=0}^l p_i e_{-\alpha_i},
\end{equation}
where $p_i=c_i z^{k_i}$ is a monomial in $z\in\bC^\ast$.  Here we assume that $c_i>0,k_i>-1$ and $p_i=p_{\nu(i)}$ for all $i$. Let us introduce the notation
\[
c=c_0^{q_0}\dots c_l^{q_l},
\quad
N=s+ \sum_{i=0}^l q_i k_i
\]
where $q_0=1$, $\psi=\sum_{i=1}^l q_i \alpha_i$, and $s=1+\sum_{i=1}^l q_i$, as usual.

As $k_i>-1$, there exists (by elementary o.d.e.\ theory) a unique holomorphic map $L$ from the (universal cover of) a punctured neighbourhood of $z=0$ to the loop group $\Lambda G=C^\infty(S^1,G)$ such that
\begin{equation}\label{L}
L^\ast \leftMC(\tfrac \partial {\partial z}) = \omega(\tfrac \partial {\partial z})   
\end{equation}
and $\lim_{z\to 0} L(z)=e\in G$, where 
$\leftMC$ is the left Maurer-Cartan form of $G$. 

With respect to the real form 
\[
\Lambda_{\bR} G =
\{
\gamma:S^1\to G
\st
\chi(\gamma(\lambda))=\gamma(1/\bar\lambda)
\}
\]
of $\Lambda G$, the loop group Iwasawa factorization
$L=L_\bR L_+$ exists on a (possibly smaller) 
neighbourhood of $z=0$.  By definition we have
\begin{gather*}
\chi(L_\bR(z,\bar z,1/{\bar\lambda}))=L_\bR(z,\bar z,\lambda)
\\
L_+(z,\bar z,\lambda)=\sum_{i=0}^\infty L_i(z,\bar z) \lambda^i,
\end{gather*}
with
$\lim_{z\to 0} L_\bR(z,\bar z,1/{\bar\lambda})=I$,
$\lim_{z\to 0} L_+(z,\bar z,1/{\bar\lambda})=I$.
Furthermore $L_0=e^{\ell(z,\bar z)}$, where $\ell(z,\bar z)\in\fh_\sharp$ for all $z$, and $\lim_{z\to 0} \ell(z,\bar z)=0$.  For these and other properties of the Iwasawa factorization we refer to \cite{PrSe86},\cite{Gu97},\cite{BaDo01}.

Let us introduce the $H$-valued \lq\lq monomial\rq\rq\  function 
\[
h(z)=e^Dz^r=e^{D+r\log z} 
\]
where $D,r\in\fh_\sharp$ are defined by 
\begin{gather*}
\alpha_i(r)=k_i+1 - \tfrac Ns,
\quad
1\le i\le l
\\
\alpha_i(D)=\log c_i - \tfrac1s\log c,
\quad
1\le i\le l
\end{gather*}
(these fix $r,D$ as $\alpha_1,\dots,\alpha_l$ are a basis of $\fh_\sharp^\ast$).
Here $H$ is the subgroup of $G$ corresponding to $\fh$.
Let 
\[
G_h=\vert h\vert\,h^{-1},
\]
where $\vert h\vert = (h\bar h)^{1/2}$. This notation is an abbreviation for
$
\bar h=e^{D+r\log \bar z},
\vert h\vert = e^{D+r\log \vert z\vert},
$
and we use the convention that \lq\lq bar\rq\rq\  means conjugation with respect to the real subspace $\fh_\sharp$ of $\fh$ or the real subgroup $h_\sharp$ of $H$.  

To relate the connection $\omega$ to the tt*-Toda equations, we introduce the $\fh_\sharp$-valued function
\[
w=\ell - \log\vert h\vert = \ell - (D+r\log\vert z\vert).
\]
Then, as in section 2 of \cite{GIL3}, one can compute
\begin{equation}\label{prelimalpha}
(L_\bR G_h)^\ast(\leftMC) =
(w_z +\tfrac1\lambda \nu \tilde E_-)dz
+
(-w_{\bar z} + \lambda \bar\nu \tilde E_+)d\bar z
\end{equation}
where 
\begin{equation}\label{nu}
\nu = (p_0^{q_0}\dots p_l^{q_l})^{1/s} 
= (cz^{N-s})^{1/s} 
\end{equation}
(the definitions of $h, G_h,w$ are motivated by this calculation).  

By construction, over the domain of definition of $L_\bR G_h$, this connection is flat.  On the other hand, direct calculation of the curvature (cf.\ Proposition \ref{zcc}) gives 
\[
2w_{z\bar z} =  - \nu\bar\nu \sum_{i=0}^l e^{-2\alpha_i(w)} H_{\alpha_i}.
\]
Thus, the Iwasawa factorization (starting with the connection form $\omega$) has produced an $\fh_\sharp$-valued function $w$ (defined locally near $z=0$) satisying a p.d.e.\ which resembles the Toda equations (\ref{Toda}).  We obtain the  Toda equations exactly if we make the change of variable $t=\frac sN c^{1/s} z^{N/s}$, for (\ref{nu}) gives $dt=\nu dz$ and then the connection form (\ref{prelimalpha}) becomes exactly the  connection form $\alpha$ of Definition \ref{alphadefinition} (but the variables $z$, $x=\vert z\vert$ there are replaced by $t$, $x=\vert t\vert$ here).

Furthermore, all conditions 
of Definition \ref{tt*-Toda} are satisfied, so we have produced a solution $w$ of the tt*-Toda equations from the connection form $\omega$, i.e.\ from the data $p_0,\dots,p_l$ where $p_i=c_i z^{k_i}$ and $c_i>0,k_i>-1$.  As in \cite{GIL3} this construction may be extended to the case $k_i\ge-1$.  

From the definition of $w$ we have
$
2w(\vert t\vert) \sim -\log \vert h\vert^2 \sim -2(\log\vert z\vert) r
$
as $z\to 0$.  Taking into account the change of variable, we deduce that 
$
2w(\vert t\vert) \sim -2m\log \vert t\vert
$
as $z\to 0$, where $m\in\fh_\sharp$ is defined by 
\begin{equation}\label{m}
\alpha_i(m)=\tfrac sN(k_i+1) - 1,
\quad
1\le i\le l.
\end{equation}
Conversely, let $w$ be any local
radial solution near zero of the p.d.e.\ 
\[
2w_{t \bar t} = -\sum_{i=0}^l k_i e^{ -2\alpha_i(w)} H_{\alpha_i}
\]
such that $2w(\vert t\vert) \sim -2m\log \vert t\vert$ as $t\to 0$.
Then 
\[
e^{ -2\alpha_i(w)} \sim e^{ 2\alpha_i(m) \log\vert t\vert }
= \vert t\vert^{2 \alpha_i(m)}
= \vert t\vert^{\tfrac {2s}N(k_i+1) - 2}.
\]
A necessary condition for such a solution to exist is that 
$\tfrac {2s}N(k_i+1) - 2\ge-2$ for all $i$, in other words
$k_i\ge-1$.  Thus, writing $z$ instead of $t$ for consistency with sections \ref{back} and \ref{eqns}, we have:

\begin{pro}\label{odepde}
Let $m\in\fh_\sharp$.  There exists a local solution near zero of the tt*-Toda equations 
such that $w(\vert z\vert)\sim -m\log\vert z\vert$ as $z\to 0$
if and only if  $\alpha_i(m)\ge-1$ for $i=0,\dots,l$.
\end{pro}

We shall compute the Stokes data of such solutions in terms of the asymptotic data $m$ (Theorem \ref{explicitQQPi}).  In view of (\ref{m}), this shows that the Stokes data depends only on the $k_i$, not on the $c_i$. 

As in \cite{GIL3} in the case $G=SL_n\bC$, we shall carry out the computation by making use of a simpler meromorphic connection form $\hat\omega$ which has the same Stokes data as $\hat\alpha$.  
We define $\ho$ by
\begin{equation}\label{hatomega}
\ho=\left(
-\tfrac sN \tfrac z{\lambda^2} \eta_- +\tfrac 1\lambda m
\right)
d\lambda.
\end{equation}
It is a meromorphic connection form with poles of order $2,1$ at $\lambda=0,\infty$.

\begin{pro}\label{gL}
Let $g=\lambda^m=e^{ (\log\lambda)m }$.  Then

(1) $(gL)^\ast \leftMC (\frac \partial {\partial z})
=\omega (\frac \partial {\partial z})$

(2) $(gL)^\ast \leftMC (\frac \partial {\partial \lambda}) = \ho (\frac \partial {\partial \lambda})$.
\end{pro}

\begin{proof}
(1) As $g$ is independent of $z$, $(gL)^\ast \leftMC (\frac \partial {\partial z}) = 
L^\ast \leftMC (\frac \partial {\partial z})$. By the definition (\ref{L}) of $L$, this is $\omega (\frac \partial {\partial z})$.  (2) Direct computation (the definition of $g$ is motivated by this).
\end{proof}

Thus we have\footnote{This implies that the combined connection $d+\omega+\ho$ is flat, and hence the monodromy data of $\ho$ is independent of $z$, just as the monodromy data of $\hat\alpha$ was independent of $z,\bar z$.  However, the isomonodromic deformation of $\ho$, given by the explicit form of $\omega$, is very simple, whereas that of $\hat\alpha$ is complicated, being given by solutions $w$ of the tt*-Toda equations.}
$(gL)^\ast \leftMC = \omega+\ho$.

The analogous statements for $\alpha,\hat\alpha$ are:

\begin{pro}\label{gLRG}
$ $

(1) $(gL_\bR G_h)^\ast \leftMC (\frac \partial {\partial z}) = 
\alpha (\frac \partial {\partial z})$,
$(gL_\bR G_h)^\ast \leftMC (\frac \partial {\partial \bar z}) = 
\alpha (\frac\partial{\partial \bar z})$ 

(2) $(gL_\bR G_h)^\ast \leftMC (\frac \partial {\partial \lambda}) = 
\hat\alpha (\frac\partial{\partial \lambda})$ 
\end{pro}

\begin{proof} (1) This follows from (\ref{prelimalpha}) and the discussion above, using the fact that $g$ is independent of $z$. (2)  This is a direct computation.
\end{proof}

This means $(gL_\bR G_h)^\ast \leftMC = 
\alpha + \hat\alpha$, so we now have an explicit formula $F=gL_\bR G_h$
for the map $F$ of section \ref{eqns} (see the remarks before and after Definition \ref{hatalpha}). 

The key relation between $\ho$ and $\hat\alpha$ is provided by Propositions \ref{gL} and \ref{gLRG} and the Iwasawa factorization $L=L_\bR L_+$.  As a consequence, it will follow that 
$\hat\alpha, \ho$ have the same Stokes data.
In keeping with the conventions of section \ref{data}, however, we shall work with $\hb$ instead of $\hat\alpha$, and the connection form 
$
\left(
-\tfrac sN \tfrac z{\lambda^2} \eta_+ +\tfrac 1\lambda m
\right)
d\lambda
$
instead of $\ho$, where $\eta_+=\sum_{i=0}^l p_i e_{\alpha_i}$.  

In classical matrix notation, this means that we are comparing the Stokes data of the two systems
\begin{gather}\label{sys1}
\tfrac{d\Psi}{d\zeta}=
\left(
-\tfrac{1}{\zeta^2}\tilde E_+ - \tfrac{1}{\zeta}xw_x + x^2 \tilde E_-
\right)
\Psi,
\\
\label{sys2}
\tfrac{d\Phi}{d\lambda}=
\left(
-\tfrac sN \tfrac z{\lambda^2} \eta_+ +\tfrac 1\lambda m
\right)
\Phi.
\end{gather}
As $\zeta=\lambda/t$ it
will be convenient to assume that $t$ and $z$ are real and positive. This allows us to use the same Stokes sectors for both systems, but does not affect the Stokes data (which is independent of $t,z$).

Recall from section \ref{data} that we have canonical solutions $\Psiz_i$ of (\ref{sys1}) on Stokes sectors
$\supersec_i$, asymptotic to the formal solution 
$\Psiz_f(\zeta)=e^{-w}(I+\sum_{k\geq 1}\psi_k \zeta^k)e^{\frac{1}{\zeta}E_+}$
of Lemma \ref{lm:formalsolution}.
Stokes factors are defined by $\Psiz_{i+\frac1s}(\zeta) =\Psiz_i(\zeta) K_{i}$.  

In exactly the same way, at the pole $\lambda=0$, we obtain  canonical solutions $\Phiz_i$ of (\ref{sys2}) on the same Stokes sectors
$\supersec_i$, asymptotic to a unique formal solution of the form 
$\Phiz_f(\lambda)=h(I+\sum_{i\ge 1} \phiz_i \lambda^i)e^{\frac t\lambda E_+}$.  
This particular formal solution arises (cf.\ the proof of Lemma \ref{lm:formalsolution}) because we have
$\eta_+=\nu hE_+ h^{-1}$ and hence
$-\tfrac sN \tfrac z{\lambda^2}\eta_+=-\tfrac t{\lambda^2}hE_+ h^{-1}$.
The formal monodromy is trivial here for the same reason as in the proof of Lemma \ref{lm:formalsolution}. 
Stokes factors are defined by $\Phiz_{i+\frac1s}(\lambda) =\Phiz_i(\lambda) J_{i}$. 

At the regular singularity $\lambda=\infty$, by standard o.d.e.\ theory, there is a canonical solution of (\ref{sys2}) of the form
$\Phii(\lambda)=(I+\sum_{i\ge1} \phii_i \lambda^{-i}) \lambda^m\lambda^M$
where $M\in\fg$ is nilpotent (we shall not need to know $M$ explicitly).
In $\lambda^m\lambda^M$ we use the (analytic continuation of the) branch of $\log\lambda$ which is real when $\lambda$ is real and positive.

\begin{pro}\label{samestokes}
We have $J_i=K_i$ for all $i$.
\end{pro}

\begin{proof}
This is analogous to the proof of Corollary 4.3 of \cite{GIL3}.
\end{proof}

\begin{pro}\label{conj}
We have
$
K_i K_{i+\frac1s}P_0=D_i \, e^{2\pi\i m/s} e^{2\pi\i M/s} P_0\,  D_i^{-1}
$
for some $D_i\in G$.
\end{pro}

\begin{proof}  From Proposition \ref{samestokes}, it suffices to prove that
$
J_i J_{i+\frac1s}P_0=D_i \, e^{2\pi\i m/s} e^{2\pi\i M/s} P_0\,  D_i^{-1}.
$
Exactly as for the case of (\ref{sys1}) in section \ref{data} (Propositions \ref{psif-symm} and \ref{psi-symm}), the Stokes analysis of (\ref{sys2})
at $\lambda=0$ gives 
\[
P_0 \Phiz_f(e^{2\pi\i/s}\lambda) P_0^{-1} = \Phiz_f(\lambda)
\]
and
\begin{equation}\label{phizcyclic}
P_0 \Phiz_{i-\frac 2s}(e^{2\pi\i/s}\lambda) P_0^{-1} = \Phiz_i(\lambda).
\end{equation}
By a similar (but easier) argument, $\Phii$ satisfies 
\begin{equation}\label{phiicyclic}
P_0 \Phii(e^{2\pi\i/s}\lambda) P_0^{-1} = \Phii(\lambda) 
e^{2\pi\i m/s} e^{2\pi\i M/s}.
\end{equation}
Since (the analytic continuations of) $\Phii$ and $\Phiz_i$ satisfy the same o.d.e., there exists some $D_i\in G$ such that $\Phii=\Phiz_iD_i$.  (Classically, $D_i$ is called the connection matrix.)    Substituting this into (\ref{phiicyclic}), we obtain
\[
P_0 \Phiz_i(e^{2\pi\i/s}\lambda) D_i P_0^{-1}
= \Phiz_i D_i  e^{2\pi\i m/s} e^{2\pi\i M/s}.
\]
By  (\ref{phizcyclic}) and the definition of $J_i , J_{i+\frac 1s}$, the left hand side is
\[
\Phiz_{i+\frac 2s}(\lambda) P_0 D_i P_0^{-1} = 
\Phiz_i(\lambda) J_i J_{i+\frac 1s}P_0 D_i P_0^{-1}.
\]
We conclude that $D_i e^{2\pi\i m/s} e^{2\pi\i M/s}=J_i J_{i+\frac 1s}P_0 D_i P_0^{-1}$,
as required. 
\end{proof}

Note that
$e^{2\pi\i m/s} P_0 = e^{ 2\pi\i (m+x_0)/s}$. 
As $M$ is nilpotent, Proposition \ref{conj} gives:

\begin{cor}\label{eigenvalues}
The semisimple part of 
$K_i K_{i+\frac1s}P_0$ is conjugate in $G$ to $e^{ 2\pi\i (m+x_0)/s}$.
\end{cor}

(By the semisimple part of an element $g$ of a linear algebraic group we mean $g^{ss}$ where  $g=g^{ss}g^{u}=g^{u}g^{ss}$ is the multiplicative Jordan-Chevalley decomposition. Here $g^{ss}$ is semisimple and $g^{u}$ is unipotent.)

Now we can give our main applications of the results in sections \ref{data} and \ref{lie}.

Recall (subsection \ref{lie2}) that we have the map
$\chi: G \to \bC^l$ and Steinberg cross-section $C^\simple$,
with $C^\simple(\bC^l)=\cR^\simple$ and $G^{\text{reg}}/G\cong \cR^\simple$.
Consider the diagram
\[
\begin{CD}
\cM_{\Delta_+^\prime} = \cR^\simple
 @>{\subseteq}>>  G^{\text{reg}}
 @>>>  G^{\text{ss}}
 \\
@.   @V{\chi}VV    @V{\chi}VV   
\\
@.  \bC^l @>{=}>> \bC^l
\end{CD}
\]
where the map $ G^{\text{reg}} \to G^{\text{ss}}$ is defined by 
$g=g^{ss}g^{u}=g^{u}g^{ss} \mapsto g^{ss}$.  It is known (\cite{S}, Theorem 3) that the induced map 
$G^{reg}/G\cong G^{ss}/G$ on the respective spaces of $G$-conjugacy classes 
is an equivalence.

\begin{thm}\label{explicitQQPi}  Let $w$ be (any) local radial solution near zero of the tt*-Toda equations, with asymptotic data $m$ at $z\to 0$. Then $M^{(0)}_{\,w}=K_1 K_{1+\frac1s}P_0$ can be computed explicitly as
$M^{(0)}_{\,w}= C^\simple(  \chi (  e^{ 2\pi\i (m+x_0)/s}  )  )$.
\end{thm} 

\begin{proof}
$M^{(0)}_{\,w}= C^\simple(  \chi ( M^{(0)}_{\,w} ))
= C^\simple(  \chi ( (M^{(0)}_{\,w})^{ss} ))
=C^\simple(  \chi (  e^{ 2\pi\i (m+x_0)/s}  )  )$.
\end{proof}

Next we consider the space of all such $M^{(0)}_{\,w}$.  Because $(M^{(0)}_{\,w})^{ss}$
is conjugate to an element of a compact real form, it is natural to consider the diagram
\[
\begin{CD}
\cM_{\Delta_+^\prime}^{\text c} = \cR^{\simple,c}
 @>{\subseteq}>>  G^{\text{reg,c}}
 @>>>  G^{\text{ss,c}}
 \\
@.   @V{\chi}VV    @V{\chi}VV   
\\
@.  \bC^l @>{=}>> \bC^l
\end{CD}
\]
where $X^{\text c}$ denotes the subspace of $X$ consisting of elements whose semisimple part is conjugate in $G$ to an element of the standard compact real form
$G_{\text{cpt}}$ (Definition \ref{compact}).

We have $G^{\text{ss,c}}/G\cong G_{\text{cpt}}/G_{\text{cpt}}$  (Lemma 5.5 of \cite{GH}).
Let us recall some well known facts concerning this space. 

First, the fundamental Weyl chamber is the convex cone consisting of points $y\in \ii\fh_\sharp$ which satisfy
the inequalities
$\alpha_i^{\text{real}}(y)\ge 0$  $(1\le i\le l)$,
where $\alpha^{\text{real}}=(2\pi\ii)^{-1}\alpha\in (\ii\fh_\sharp)^\ast$ is the real root corresponding to the (complex) root $\alpha$.  
The fundamental Weyl alcove $\fA$
is the convex polytope defined by the inequalities
\begin{equation}\label{inequalities}
\alpha_i^{\text{real}}(y)\ge 0 \ (1\le i\le l), \ \psi^{\text{real}}(y)\le 1.
\end{equation}
When $G$ is simply-connected (as we are assuming), this convex polytope parametrises the conjugacy classes of $G_{\text{cpt}}$, i.e.\ we have $ G_{\text{cpt}}/G_{\text{cpt}}\cong\fA$.

Let us denote the
$G$-conjugacy class of $A_0\in G_{\text{cpt}}$ by $[A_0]$.  Then, for any 
$A\in[A_0]$, there is a unique $y\in\ii\fh_\sharp$ satisfying (\ref{inequalities}) such that $[A]=[e^{2\pi\i y}]$.

\begin{pro}
For $A=(M^{(0)}_{\,w})^{\text{ss}}$ we have $y=(m+x_0)/s$.
\end{pro}

\begin{proof}  From Proposition \ref{odepde}, $m$ is characterized by the conditions
\[
\text{$\alpha_i(m)\ge -1$ for $i=0,\dots,l$. }
\]
First we consider  $i=1,\dots,l$.
As $\alpha_i(x_0)=1$ here, we have
\begin{align*}
\alpha_i(m)\ge -1 &\iff  \alpha_i(m+x_0)\ge 0
\\
&\iff  \alpha_i(  (m+x_0)/s  )\ge 0
\\
&\iff  \alpha_i^{\text{real}}(2\pi\ii(m+x_0)/s)\ge 0.
\end{align*}
Next we consider $\alpha_0=-\psi$.  As
$\psi(x_0)= \sum_{i=1}^l q_i \alpha_i(x_0) = \sum_{i=1}^l q_i = s-1$, we have
\begin{align*}
\alpha_0(m)\ge -1 &\iff  \psi(m)\le 1
\\
&\iff  \psi(m+x_0)\le s
\\
&\iff  \psi( (m+x_0)/s )\le 1
\\
&\iff \psi^{\text{real}}(2\pi\ii(m+x_0)/s)\le 1.
\end{align*}
From these two calculations we see that $y=(m+x_0)/s$.
\end{proof}

This gives our description of the space of Stokes data:

\begin{thm}\label{convexset} 
We have a natural bijection between

(a) the space of Stokes data $M^{(0)}_{\,w}$ for
local radial solutions near zero of the tt*-Toda equations, and

(b)  the fundamental Weyl alcove $\fA$ of $G_{\text{cpt}}$ when
$\fg\ne \fa_l,\fd_{2m+1},\fe_6$, and the set
 $\fA^\sigma =\{ y\in\fA \st \sigma(y)=y \}$
when $\fg= \fa_l,\fd_{2m+1},\fe_6$.
\end{thm} 

The space of asymptotic data
\[
\cA=\{ m\in \fh_\sharp \st 
\alpha_i(m)\ge -1, i=0,\dots,l \}
\]
(and $\cA^\sigma$)
is, of course, also a convex polytope, although it depends on a specific way of writing the tt*-Toda equations.  Our map 
\[
\cA\to\fA,\quad
m\mapsto \tfrac{2\pi \i}{s}(m+x_0)
\]
identifies it naturally with the fundamental Weyl alcove.

\appendix
\section{Formal solutions}\label{sec:formalsolution}

This section is included primarily as motivation for Lemma \ref{lm:formalsolution}. We sketch a proof of the existence of a formal solution to the equation
\begin{equation}\label{eq:irregular}
\frac{d\Psi}{d\xi}=
\left(
\sum_{k=-r-1}^\infty A_{k+1} \xi^k
\right) 
\Psi
\end{equation}
following \cite{FIKN06}, Proposition 1.1, but using Lie-theoretic notation as far as possible.  Here $r\in\bN$ and all coefficients $A_k$ belong to $\fg$, which we take to be the Lie algebra of a (simple) matrix Lie group $G$. 

We assume that the leading coefficient $A_{-r}$ is regular, hence contained in a unique Cartan subalgebra $\fh_1$ (thus we have $\beta(A_{-r})\ne0$ for all roots $\beta\in\Delta_1$ with respect to $\fh_1$).
 
 Let $P\in G$.  Then $\Lambda_{-r}=\Ad(P^{-1})A_{-r}$ is also regular, and contained in a unique Cartan subalgebra $\fh_2$ (namely $\Ad(P^{-1})\fh_1$). Let $\fg=\fh_2\oplus(\oplus_\alpha \fg_\alpha)$ be the eigenspace decomposition for the adjoint action of $\fh_2$.

\begin{pro}\label{pro:formalsolution}
There is a unique formal fundamental solution $\Psi_f$ of equation 
(\ref{eq:irregular}) of the form
\begin{equation}\label{eq:solution}
\Psi_f(\xi)=P
\left(
\sum_{k=0}^\infty \psi_k\xi^k
\right)
\exp^{\Lambda(\xi)}, 
\end{equation}
where $\psi_0=I$ and  $\displaystyle\Lambda(\xi)=\Lambda_0\log \xi+\sum_{k=-r}^{-1}\frac{\Lambda_{k}}{k}\xi^{k}$ with all $\Lambda_k\in\fh_2$.
\end{pro}

\begin{proof} Consider 
\begin{equation}\label{eq:test}
\Psi_f(\xi)=P
\left(
\sum_{k=0}^\infty Y_k \xi^k
\right)
\exp^{\Upsilon(\xi)}.
\end{equation} 
We shall show that it is possible to find $Y_k$ and $\Upsilon(\xi)=\Lambda(\xi)+\sum_{k=1}^\infty \frac{\Lambda_k}{k}\xi^k$ with $Y_0=I$, all $Y_k\in \oplus_\alpha\fg_\alpha$,
and all $\Lambda_k\in\fh_2$, such that (\ref{eq:test}) formally satisfies (\ref{eq:irregular}). As (\ref{eq:test}) can be rewritten in the required form (\ref{eq:solution}),
this will be sufficient.

Substitution into (\ref{eq:irregular}) leads to the recurrence relations
\begin{equation}\label{eq:relationY}
\Lambda_{-r+k}+[Y_k,\Lambda_{-r}]=F_{-r+k}
\end{equation}
for $k\in\bN$, 
where $F_{-r+1}=P^{-1}A_{-r+1}P$ and, for $k\ge2$,
\begin{eqnarray*}
F_{-r+k}&=&P^{-1}A_{-r+k}P+\sum_{n=1}^{k-1}(P^{-1}A_{-r+k-n}PY_n-Y_n\Lambda_{-r+k-n})\\
&&+\left\{\begin{array}{ll}0&k=2,3,\dots,r \\ -(k-r)Y_{-r+k} & k=r+1,r+2,\dots
\end{array}\right.
\end{eqnarray*}
Taking components of (\ref{eq:relationY}) in
$\fg^\prime=\fh_2$,
$\fg^{\prime\prime}=\oplus_\alpha \fg_\alpha$
we obtain
\begin{equation*}
\Lambda_{-r+k}=\text{Proj}^\prime(F_{-r+k}),\quad
[Y_k,\Lambda_{-r}]=\text{Proj}^{\prime\prime}(F_{-r+k}).
\end{equation*}
As $\Ad \,\Lambda_{-r}$ is invertible on $\oplus_\alpha \fg_\alpha$,  all of the 
$\Lambda_k$, $Y_k$ can be determined recursively.  
\end{proof}

\section{Singular directions and the (enhanced) Coxeter Plane}\label{sec:coxeterplane}

Let $\fg$ be a complex simple Lie algebra with $\rank\,\fg = l>1$. Let $\fh$ be a Cartan subalgebra, and $\Delta_+$ a choice of positive roots. The vectors $H_\alpha\in\fh$ are defined as in section \ref{back} by
$B(h,H_\alpha)=\alpha(h)$. They span a polyhedron in the vector space $\fh$.
The Coxeter Plane consists of a certain real two-dimensional subspace of $\fh$ together with the
orthogonal projections of all points $H_\alpha\in\fh$
onto this plane.  The rays from the origin to these points (the \lq\lq spokes\rq\rq), and/or the concentric circles passing through these points 
(the \lq\lq wheels\rq\rq) may be drawn.  For the Lie algebra $\fe_8$ this picture is very well known and appears as the frontispiece of Coxeter's book \cite{Co74}.
\begin{figure}[h]
\scalebox{0.7}{\includegraphics{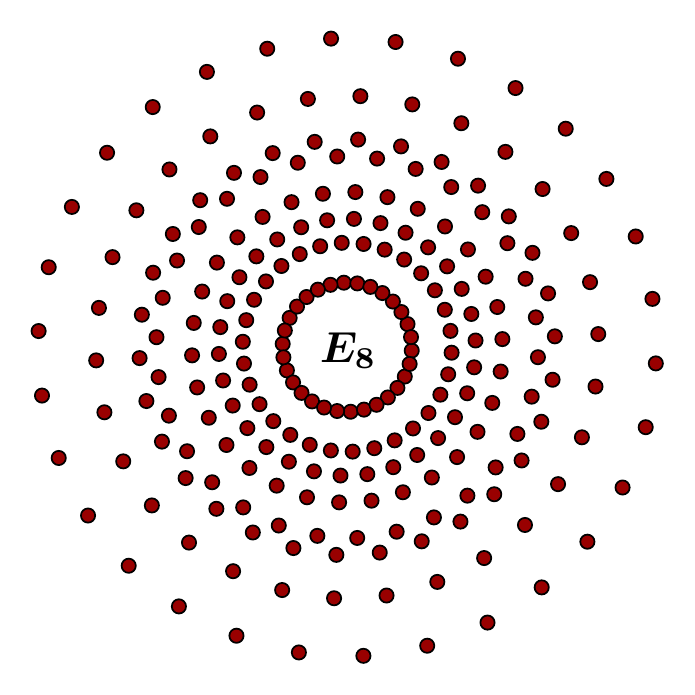}}
\caption{Coxeter Plane with root projections for $\fe_8$}\label{E8}
\end{figure}
The version in Figure \ref{E8} is taken from \cite{Ca17}. There are $60$ spokes and $8$ wheels in this case.

We shall sketch three descriptions of the Coxeter Plane, all of them rather indirect. Then we shall explain how the Stokes data of the meromorphic o.d.e.\ of section \ref{data} (or that of section \ref{local}) provides a more direct description of the Coxeter Plane, which also illustrates its properties effectively.

The first (and, as far as we know, original) description is aesthetic:  the two-dimensional plane is chosen to give the \lq\lq most symmetrical projection\rq\rq. 

We take the second description from Kostant's article \cite{Kostant10}.  It depends on the choice of the Lie algebra element $E_+$, hence the choice of a Cartan subalgebra $\fh^\prime=\fg^{E_+}$ in apposition to $\fh$. We have the real subspace
\[
\fh^\prime_\sharp=\{ X\in\fh^\prime \st \beta(X)\in\bR
\ \text{for all}\ \beta\in\Delta^\prime\}.
\]
Recall that the Coxeter number of $\fg$ is 
$s=1+\sum_{i=1}^l q_i = \sum_{i=0}^l q_i$, where
$\psi=\sum_{i=1}^l q_i \alpha_i$ and $q_0=1$, and that $\tau=\Ad P_0$ acts on $\fh^\prime$
as a Coxeter element, with $\tau^s=1$. 
Let us choose (for simplicity) the specific coefficients
$E_+=\sum_{i=0}^l \sqrt{q_i} e_{\alpha_i}$.  Then $E_+\in \fh^\prime=\fh^\prime_\sharp\otimes\bC$ is isotropic with respect to $B$, hence defines an oriented real two-dimensional subspace $Y$ of 
$\fh^\prime_\sharp$. This is the Coxeter Plane  (with respect to the Cartan subalgebra  $\fh^\prime$). Let 
$Q:\fh^\prime_\sharp\to Y$ be orthogonal projection with respect to the (positive definite) inner product $B|_{\fh^\prime_\sharp}$.  
Then (\cite{Kostant10}, section 0.2) $Y$ is the essentially unique plane with the property that the projection $Q$ commutes with the action of the Coxeter element $\tau$. This gives a precise meaning to \lq\lq most symmetrical projection\rq\rq.

The third (more abstract) description is based on the theory of Coxeter groups (\cite{Steinberg},\cite{H},\cite{Ca17}), and in our situation this means the Weyl group of $\fg$ with respect to a Cartan subalgebra $\fh$.  We sketch this theory, referring to \cite{Ca17} and section 3.19 of \cite{H} for further details. Any product
\[
\tau=s_1\cdots s_l
\]
of reflections in simple roots is called a Coxeter element. We may assume that
\[
\tau = xy,\quad x=s_1\cdots s_k, y=s_{k+1}\cdots s_l
\]
for some $k$, where $s_1,\dots,s_k$ commute and $s_{k+1},\dots,s_l$ commute. Thus
$x^2=1$ and $y^2=1$ and $x,y$ generate a dihedral group, in fact the unique
dihedral subgroup of the Weyl group which contains $\tau$.
The subspaces 
\[
\Ker\,\alpha_1\cap\cdots\cap\Ker\,\alpha_k,
\Ker\,\alpha_{k+1}\cap\cdots\cap\Ker\,\alpha_l
\]
contain distinguished real lines $l_x,l_y$ which can be constructed explicitly from 
a certain \lq\lq Perron-Frobenius eigenvalue\rq\rq\  of the Cartan matrix. The Coxeter Plane is the real span of $l_x,l_y$.
The lines of the Coxeter Plane are given by taking the orthogonal complements of the intersections of all root hyperplanes 
$\Ker\,\alpha_i$ with this plane. 

It can be shown that $\tau$ acts on this plane as rotation through $2\pi/s$, and $x,y$ act by reflection in 
adjacent lines (given by the intersections of 
$\Ker\,\alpha_1\cap\cdots\cap\Ker\,\alpha_k,
\Ker\,\alpha_{k+1}\cap\cdots\cap\Ker\,\alpha_l$
with the plane).  It follows that there are exactly $s$ equally spaced lines in the Coxeter Plane,
whose symmetry group is the dihedral group generated by $x$ and $y$. 

As explained in \cite{Ca17}, the relation with the Cartan matrix leads to an efficient algorithm for drawing this version of the Coxeter Plane. We are grateful to Bill Casselman for the current version of \cite{Ca17} which contains these pictures.

It turns out that the Stokes data of the meromorphic o.d.e.\ of section \ref{data} provides a fourth description of the Coxeter Plane.  To see this, we observe that the diagram of singular directions is related to Kostant's plane $Y$, as follows. First we introduce the notation
\[
E_+=\Re(E_+)+\ii\Im(E_+),\quad \Re(E_+),\Im(E_+)\in
\fh^\prime_\sharp
\]
for the decomposition of $E_+$ with respect to the real subspace 
$\fh^\prime_\sharp$.

\begin{pro}\label{ident}
Let us identify Kostant's plane $Y$ with $\bC$ by identifying the orthonormal basis 
$(2/s)^{\frac12}\Re(E_+),(2/s)^{\frac12}\Im(E_+)$
with $1,\i$.  Then the points $Q(H_{\beta})$ are identified with the points $(2/s)^{\frac12}\beta(E_+)$, for all $\beta\in\fh^\prime$. 
\end{pro}

\begin{proof}  We use the key fact that the complex conjugate of $E_+$ with restect to $\fh^\prime_\sharp$ is given by
$\bar E_+ = E_-$ (\cite{Kostant10}, Theorems 1.11 and 1.12). Then by direct calculation we obtain
\begin{align*}
&B(\Re(E_+),\Re(E_+))=\tfrac{s}{2}
\\
&B(\Im(E_+),\Im(E_+))=\tfrac{s}{2}
\\
&B(\Re(E_+),\Im(E_+))=0
\end{align*}
(so $E_+$ is isotropic, as stated earlier). 
The projection to $Y$ of a vector $X\in\fh^\prime_\sharp$ is
\[
Q(X)=\tfrac2s\, B(X,\Re(E_+))\Re(E_+)+
\tfrac2s\, B(X,\Im(E_+))\Im(E_+).
\]
In particular,
\begin{equation}\label{qhbeta}
Q(H_\beta)=\tfrac2s\, \beta(\Re(E_+))\Re(E_+)+
\tfrac2s\, \beta(\Im(E_+)) \Im(E_+)
\end{equation}
for any $\beta\in\Delta^\prime$.  Note that
$\beta(\Re(E_+))$ and $\beta(\Im(E_+))$ are both real.
It follows that the vector $Q(H_\beta)$ in the plane $Y$ --- under the identification of
$(2/s)^{\frac12}\Re(E_+),(2/s)^{\frac12}\Im(E_+)$ in $Y$
with $1,\i$ in $\bC$ --- corresponds to the complex number 
$(2/s)^{\frac12}\beta(\Re(E_+))+\ii (2/s)^{\frac12}\beta(\Im(E_+))$, and this is just
$(2/s)^{\frac12}\beta(E_+)$.
\end{proof}

In the proposition and its proof we have assumed that $E_+=\sum_{i=0}^l \sqrt{q_i} e_{\alpha_i}$. In the general case $E_+=\sum_{i=0}^l c^+_i e_{\alpha_i}$, (\ref{qhbeta}) becomes
\begin{equation*}
Q(H_\beta)=\tfrac{2t}s\, \beta(\Re(E_+))\Re(E_+)+
\tfrac{2t}s\, \beta(\Im(E_+)) \Im(E_+)
\end{equation*}
where $t\in\bC^\ast$ is defined by
$t( \Re(E_+)-\i \Im(E_+) ) = 
\sum_{i=0}^l ({q_i}/{c^+_i}) e_{-\alpha_i}$
(\cite{Kostant10}, Theorem 1.11). From this we obtain
$Q(H_{\beta})=((2t)/s)^{\frac12}\beta(E_+)$.  Thus the  diagram of $2s$ singular directions, together with the points 
$\beta(-E_+)$ marked with their corresponding roots $\beta\in\Delta^\prime$, is --- up to scaling and rotation --- simply the Coxeter Plane.
In particular, using the fact that the symmetry group of the Coxeter Plane is a dihedral group of order $2s$, we see that Theorem \ref{thm:singulardirection} holds for any such $\fg$, not just for the classical Lie algebras. 

Our diagram of singular directions may in fact be regarded an \lq\lq enhanced Coxeter Plane\rq\rq, because of the additional choice of $E_+$, which fixes a Cartan subalgebra in apposition and a particular Coxeter element.
This facilitates the assignment of a root $\beta$ to a ray in the Coxeter Plane: one simply takes the ray through the point $\beta(-E_+)$.  
Having made this assignment, the Coxeter element acts on the roots by clockwise rotation through $2\pi/s$. There are $2s$ systems of positive roots corresponding to this Coxeter element (in the sense of Proposition \ref{pro:P0asAgamma}), given by the $2s$ choices of \lq\lq positive sectors\rq\rq.  The head and tail of a positive sector give the associated simple roots. The roots on any two consecutive rays generate all the roots. While this information may be implicit in the (usual) Coxeter Plane, the choice of $E_+$ makes it explicit.

In the other direction, as our diagram of singular directions arises independently of Lie theory from the Stokes data of the meromorphic connection $d+\hat\alpha$ (or $d+\hat\omega$), the possibility of reproving some classical results on root systems arises.  This would be very much in the spirit of \cite{B_inv},
where the Stokes data of a  meromorphic connection 
(similar to $d+\hat\omega$)
was used rather unexpectedly to establish a result in symplectic geometry, the Ginzburg-Weinstein isomorphism.

\end{document}